\documentclass[11pt,a4paper]{amsart}
\usepackage{amssymb,amsmath}
\usepackage[all]{xy}

\theoremstyle{plain}
\newtheorem{theorem}{Theorem}[section]
\newtheorem{lemma}[theorem]{Lemma}
\newtheorem{proposition}[theorem]{Proposition}
\theoremstyle{definition}
\newtheorem{definition}[theorem]{Definition}
\newtheorem{problem}[theorem]{Problem}
\theoremstyle{remark}
\newtheorem{remark}[theorem]{Remark}

\DeclareMathOperator{\Spec}{Spec}
\DeclareMathOperator{\chr}{char}
\newcommand{\C}{\mathbb{C}}
\newcommand{\R}{\mathbb{R}}
\newcommand{\Q}{\mathbb{Q}}
\newcommand{\Z}{\mathbb{Z}}
\renewcommand{\mod}{\mathrm{mod}\:}

\begin{document}
\title[Gorenstein isolated quotient singularities]{Gorenstein isolated 
       quotient \\
       singularities over $\C$}
\author{D.~A. Stepanov}
\address{The Department of Mathematical Modelling \\
         Bauman Moscow State Technical University \\
         2-ya Baumanskaya ul. 5, Moscow 105005, Russia}
\email{dstepanov@bmstu.ru}
\thanks{The research was supported by the Russian Grant for Scientific 
        Schools 1987.2008.1 and by the Russian Program for Development of
        Scientific Potential of the High School 2.1.1/227}
\date{}
\begin{abstract}
In this paper we review the classification of isolated quotient
singularities over the field of complex numbers due to H. Zassenhaus,
G. Vincent, and J.~A. Wolf. As an application we describe Gorenstein
isolated quotient singularities over $\C$, generalizing a result of 
K. Kurano and S. Nishi.
\end{abstract}
\maketitle

\section{Introduction}\label{S:intro}
Let $\Bbbk$ be a field, and $R$ a finitely generated $\Bbbk$-algebra.
Recall that $R$ is called \emph{Gorenstein} if it has a canonical module
generated by one element (see, e.~g., \cite{Eisenbud}, 21.1). 
Geometrically, the affine algebraic variety $X=\Spec R$ is Gorenstein
if the canonical divisor $K_X$ of $X$ is Cartier. Now let $\Bbbk=\C$ be the 
field of complex numbers, $S=\C[x_1,\dots,x_N]$ the polynomial ring in $N$ 
variables, $G$ a finite subgroup of $GL(N,\C)$, and $R=S^G$ the algebra of 
invariants. By Hilbert Basis Theorem $R$ is finitely generated. We say 
that $X=\Spec R$ is an \emph{isolated singularity}, if the algebraic 
variety $X$ is singular and its singular set has dimension $0$.

Our paper is devoted to the following
\begin{problem}\label{Pm:classgisqsing}
Classify all Gorenstein isolated quotient singularities of the form
$\C^N/G$, where $G$ is a finite subgroup of $GL(N,\C)$, up to an 
isomorphism of algebraic varieties.
\end{problem}

Recently K. Kurano and S. Nishi obtained the following result.
\begin{theorem}[\cite{KN}, Corollary~1.3]\label{T:KN}
Let $G$ be a finite subgroup of $GL(p,\C)$, where $p$ is an odd prime,
and assume that $\C^p/G$ is an isolated Gorenstein singularity. Then 
$\C^p/G$ is isomorphic to a cyclic quotient singularity, i.~e., to
a singularity $\C^p/H$, where $H<GL(p,\C)$ is a cyclic group.
\end{theorem}
Kurano and Nishi give a nice direct proof of their theorem. On the other
hand, there exists a complete classification of isolated quotient
singularities over $\C$, and it is natural to try to derive results of
this sort and even a complete description of Gorenstein isolated
quotient singularities over $\C$ from this classification.

Isolated quotient singularities (IQS in the sequel) over $\C$ were 
classified mainly in works of H. Zassenhaus \cite{Zassenhaus} (1935),
G. Vincent \cite{Vincent} (1947), and J.~A. Wolf \cite{Wolf} (1972). 
These authors were motivated neither by singularity theory nor by
algebraic geometry. In his pioneer work Zassenhaus studied nearfields,
whereas Vincent and Wolf classified compact Riemannian manifolds of
constant positive curvature. The final results were obtained by Wolf
and stated in \cite{Wolf} as a classification of such manifolds. The
fact that Zassenhaus-Vincent-Wolf classification gives also a 
classification of IQS has already been known for a long time (see, e.~g.,
\cite{Popov}), but it seems that this topic has not attracted much of
the attention of algebraic geometers yet.

An element $g\in GL(N,\C)$ is called a \emph{quasireflection}, if $g$ fixes
pointwise a codimension $1$ linear subspace of $\C^N$. Gorenstein IQS over 
$\C$ are characterized by the following theorem of K.-I. Watanabe.
\begin{theorem}[\cite{Watanabe}]\label{T:Watanabe}
Let $G$ be a finite subgroup of $GL(N,\C)$ free from quasireflections.
Then $\C^N/G$ is Gorenstein if and only if $G$ is contained in $SL(N,\C)$.
\end{theorem}
It follows that in order to get a classification of Gorenstein IQS one
has to look through Zassenhaus-Vincent-Wolf groups giving IQS and check
which of them are contained in the special linear group. We perform this
program in our paper and present the results in 
Theorem~\ref{T:classgiqsing}.

Chapters 5, 6, and 7 of Wolf's book \cite{Wolf} are our main reference.
The presentation there is so clear, detailed and self-contained that it
is hard to imagine it to be improved. But the classification is done there
from the point of view of the Riemannian geometry. Also, anyway we have
to introduce a lot of notation and this requires a considerable place.
Therefore we decided to include to our work a review of 
Zassenhaus-Vincent-Wolf classification written from the point of view of
the theory of singularities.

Our paper is organized as follows. In Section~\ref{S:prelim} we list
some preliminaries, observations and general results on IQS and their
classification. In particular, we explain the connection of IQS with
classification of compact Riemannian manifolds of constant positive
curvature. In Section~\ref{S:class} we present the classification
of isolated quotient singularities over $\C$ as it is given in \cite{Wolf}.
The results are summarized in Theorem~\ref{T:classisqsing}. The only new 
thing we add is the computation of determinants of irreducible 
representations which we use later in description of Gorenstein
singularities. We also discuss which part of the classification extends
more or less obviously to other fields. In Section~\ref{S:Gorenstein}
we show how the result of Kurano and Nishi can be deduced and generalized
using the classification. As an example we give a description of Gorenstein
isolated quotient singularities over $\C$ in dimensions $N\leq 7$.

We hope that the beautiful classification of Zassenhaus, Vincent, and Wolf
will be completed over as many different fields $\Bbbk$ as possible and 
will find other interesting applications in algebraic geometry.

This paper originates from the talk which was given by the author at
the workshop ``Quotient singularities'' organized by Ivan Cheltsov in June
2010 at the University of Edinburgh. We thank the organizer and all the
participants of the workshop for warm and stimulating atmosphere.

\section{Preliminaries}\label{S:prelim}
The results of this section are either well known or easy, so even if we
state some of them without reference, there is no claim of originality.

\subsection{Generalities on isolated quotient singularities. Groups without 
            fixed points}
We start with formulating general problems on classification of IQS. 
Let $V$ be an algebraic variety defined over a field $\Bbbk$. Let $G$ be a 
finite group acting on $V$ by automorphisms and $P\in V$ a closed 
nonsingular fixed point of this action. Denote by $\pi\colon V\to V/G$ the 
canonical projection to the quotient variety and let $Q=\pi(P)$. Assume 
that $Q$ is an isolated singular point of $V/G$.
\begin{problem}\label{Pm:classisqsing}
Classify all singularities $Q\in V/G$ of the form described above up to
a formal or, when $\Bbbk=\C$, up to an analytic equivalence.
\end{problem}

At a nonsingular point $P$ any algebraic variety is formally equivalent
to $0\in \Bbbk^N$, $N=\dim V$, i.~e., we have an isomorphism of complete
local rings $\widehat{\mathcal{O}_{P,V}}\simeq
\widehat{\mathcal{O}_{0,\Bbbk^N}}$ (\cite{AMD}, Remark~2 after 
Proposition~11.24). With the additional condition that characteristic
of $\Bbbk$ does not divide order of $G$, the action of $G$ can be formally
linearized at $P$ (see Lemma~\ref{L:linearization} below). In the analytic
case $\Bbbk=\C$, $V$ is a manifold at $P$ and the action of $G$ can be
linearized in some local analytic coordinates at $P$, see, e.~g.,
\cite{Akhiezer}, p. 35. Thus we see that if $\chr\Bbbk \nmid |G|$,
Problem~\ref{Pm:classisqsing} is equivalent to the following
\begin{problem}\label{Pm:classaffquotients}
Classify all isolated quotient singularities of the form $\Bbbk^N/G$, 
where $G$ is a finite subgroup of $GL(N,\Bbbk)$.
\end{problem}

Let us prove a lemma on formal linearization. 
\begin{lemma}\label{L:linearization}
Let $R=\Bbbk[[x_1,\dots,x_N]]$ be the local ring of formal power series in
$N$ variables over a field $\Bbbk$, and $G$ a finite group acting on $R$ by
local automorphisms. Assume that characteristic of the field $\Bbbk$ does 
not divide order of $G$. Then one can choose new local parameters $y_1,
\dots,y_N$ in $R$ such that $G$ acts by linear substitutions in $y_1,
\dots,y_N$.
\end{lemma}
\begin{proof}
Recall that $R=\Bbbk[[x_1,\dots,x_N]]$ can be identified with the inverse 
limit $\varprojlim R/\mathfrak{m}^n$, where $\mathfrak{m}$ is the maximal
ideal $(x_1,\dots,x_N)$. Since $G$ acts by local automorphisms, it
preserves all the ideals $\mathfrak{m}$, $\mathfrak{m}^2,\dots,
\mathfrak{m}^n,\dots$, and therefore acts linearly on all quotients
$\mathfrak{m}/\mathfrak{m}^n$, $n\geq 1$, which in turn are finite 
dimensional vector spaces over $\Bbbk$. Moreover, if $\pi_{nn'}$, 
$n\geq n'$, denotes the natural truncation map $\mathfrak{m}/
\mathfrak{m}^{n+1}\to\mathfrak{m}/\mathfrak{m}^{n'+1}$, the action of $G$ 
is compatible with $\pi_{nn'}$. We shall construct $y_1,\dots,y_N$ as 
limits of Cauchy sequences $y_{1}^{(n)},\dots$, $y_{N}^{(n)}$, where 
$\forall i=1,\dots,N$, $y_{i}^{(n)}\in \mathfrak{m}/\mathfrak{m}^{n+1}$ and 
$\pi_{nn'}(y_{i}^{(n)})=y_{i}^{(n')}$.

Let $y_{1}^{(1)},\dots,y_{N}^{(1)}$ be any basis of $\mathfrak{m}/
\mathfrak{m}^2$. Note that $G$ acts on $y_{1}^{(1)},\dots,y_{N}^{(1)}$ 
by linear substitutions. Now suppose that $y_{1}^{(n)},\dots,y_{N}^{(n)}$
have already been constructed. Due to our assumption on characteristic of 
$\Bbbk$, the representation of $G$ on $\mathfrak{m}/\mathfrak{m}^{n+2}$ is
completely reducible. Thus we have a decomposition $\mathfrak{m}/
\mathfrak{m}^{n+2}\simeq \mathfrak{m}^{n+1}/\mathfrak{m}^{n+2}\bigoplus V$,
where $V$ projects isomorphically onto $\mathfrak{m}/\mathfrak{m}^{n+1}$
under the map $\pi_{n+1,n}$. Now we can uniquely lift 
$y_{1}^{(n)},\dots,y_{N}^{(n)}$ to some elements 
$y_{1}^{(n+1)},\dots,y_{N}^{(n+1)}$ of $\mathfrak{m}/\mathfrak{m}^{n+2}$.
By construction $G$ acts on $y_{1}^{(n+1)},\dots,y_{N}^{(n+1)}$ by linear
substitutions, hence the same holds for the limits $y_1,\dots,y_N$
of the sequences $y_{1}^{(n)},\dots,y_{N}^{(n)}$.
\end{proof}

From now on, if not stated otherwise, we work only with the case when the 
field $\Bbbk$ is algebraically closed and $\chr\Bbbk\nmid |G|$. Let us 
consider Problem~\ref{Pm:classaffquotients}. Recall that an element 
$g\in GL(N,\Bbbk)$ is called a \emph{quasireflection}, if $g$ fixes 
pointwise a codimension $1$ subspace of $\Bbbk^N$. By 
Chevalley-Shephard-Todd Theorem (\cite{Benson}, Theorem~7.2.1), a quotient 
$\Bbbk^N/H$, where $H<GL(N,\Bbbk)$ is finite and $\chr\Bbbk\nmid |H|$, is 
smooth if and only if the group $H$ is generated by quasireflections. 
Quasireflections contained in a finite group $G<GL(N,\Bbbk)$ generate a 
normal subgroup $H\vartriangleleft G$. It follows that from the point of 
view of Problem~\ref{Pm:classaffquotients} we can restrict ourselves to
groups $G<GK(N,\Bbbk)$ free from quasireflections. We require also the
singularity $\Bbbk^N/G$ to be isolated. This imposes additional strong
restrictions on the group $G$.
\begin{lemma}[(cf. \cite{KN}, p. 2, 3)]\label{L:eigen1}
Let $G$ be a finite subgroup of $GL(N,\Bbbk)$. Assume that $G$ is free from
quasireflections and $\chr \Bbbk\nmid |G|$. Then $\Bbbk^N/G$ has isolated 
singularities if and only if $1$ is not an eigenvalue of any element of
$G$ except the unity.
\end{lemma}
\begin{proof}
Sufficiency of the condition given in the lemma is clear, so let us prove
necessity. Suppose on the contrary that $G$ has some non-unit element
with eigenvalue $1$. Let $\mathcal{U}$ be the set of linear subspaces
of $\Bbbk^N$ with nontrivial stabilizers in $G$, and let $U$ be a maximal
element in $\mathcal{U}$ with respect to inclusion. Denote by $H$ the
stabilizer of $U$. If $U$ is positive dimensional, the given representation
of $H$ is reducible, so let us consider the splitting 
$\Bbbk^N=U\bigoplus U'$, where $U'$ is the complementary invariant subspace 
of $H$. Let $\overline{U}$ be the image of $U$ under the canonical 
projection $\pi\colon \Bbbk^N\to \Bbbk^N/G$. At a general point $P$ of 
$\overline{U}$ the quotient $P\in \Bbbk^N/H$ is formally isomorphic to the 
direct product $(0,\overline{0})\in U\times (U'/H)$. On the other hand, 
$\Bbbk^N/G$ has to be nonsingular at $P$, thus, by Shephard-Todd Theorem, 
the group $H$ acting on $U'$ is generated by quasireflections. But then 
also $H$ acting on $\Bbbk^N$ and thus $G$ contain quasireflections, which 
contradicts the conditions of the lemma.
\end{proof}
\begin{definition}\label{D:grpswithoutfps}
Following \cite{Wolf}, we call a group $G$ satisfying the conditions
of Lemma~\ref{L:eigen1} (i.~e. $1$ is not an eigenvalue of any element
of $G$ except the unity) a \emph{group without fixed points}.
\end{definition}

\subsection{Clifford-Klein Problem. $pq$-conditions}
To introduce the context in which a solution for 
Problem~\ref{Pm:classaffquotients} was first obtained, set $\Bbbk=\C$. It 
is a standard fact of the representation theory that any finite subgroup
of $GL(N,\C)$ is conjugate to a subgroup of the unitary group $U(N)$. On
the other hand conjugate groups $G$ and $G'$ obviously give isomorphic
quotient singularities $\C^N/G$ and $\C^N/G'$. Thus we may assume from
the beginning that $G$ is a subgroup without fixed points of $U(N)$. If
we equip $\C^N$ with the standard Hermitian product and the corresponding 
metric, we see that $G$ acts by isometries on the unit sphere $S^{2N-1}$
of $\C^N$. Moreover, $G$ acts on $S^{2N-1}$ without fixed points in the 
usual sense, which justifies Definition~\ref{D:grpswithoutfps}. Further,
the quotient $S^{2N-1}/G$ with the induced metric is a compact Riemannian
manifold of constant positive curvature. We immediately get
\begin{proposition}
Let $P\in X$ be an isolated quotient singularity over the field $\C$. Then
the link of $P$ in $X$ can be equipped with a Riemannian metric of constant
positive curvature.
\end{proposition}
Classification of Riemannian manifolds of constant curvature was an
important field of research in XX century. A key result about manifolds
of positive curvature is the following fundamental theorem.
\begin{theorem}[Killing, Hopf; see \cite{Wolf}, Corollary~2.4.10]
Let $M$ be a Riemannian manifold of dimension $n\geq 2$, and $K>0$ a
real number. Then $M$ is a compact connected manifold of constant curvature
$K$ if and only if $M$ is isometric to a quotient space of the form
$$S_{K}^{n}/\Gamma\,,$$
where $S_{K}^{n}$ is the standard $n$-sphere of radius $1/K^2$ in 
$\R^{n+1}$, and $\Gamma$ is a finite subgroup of $O(n+1)$ acting on 
$S_{K}^{n}$ without fixed points.
\end{theorem}
The problem of classification of such spaces $S_{K}^{n}/\Gamma$ was called 
by Killing the \emph{Clifford-Klein Space Form Problem}. In the sequel we 
simply call it the \emph{Clifford-Klein Problem}.

It follows from the discussion above that Problem~\ref{Pm:classaffquotients} 
for $\Bbbk=\C$ is a part of the Clifford-Klein Problem. But in fact they 
are almost equivalent. Indeed, it is shown in \cite{Wolf}, Section~7.4, 
that for even $n$ we have only $2$ solutions for the Clifford-Klein 
Problem: the sphere $S^n$ ($\Gamma=1$) and the real projective space 
$\mathbb{RP}^n$ ($\Gamma=\Z/2$). For odd $n$ all needed subgroups $\Gamma$ 
of $O(n+1)$ (orthogonal representations) are obtained from the subgroups of 
$U((n+1)/2)$ (unitary representations) without fixed points by means of
the standard representation theory (Frobenius-Schur Theorem, see 
\cite{Wolf}, Theorem~4.7.3).

Certainly, the problem of classification of finite subgroups 
$G<GL(N,\Bbbk)$ is wider than the problem of classification of IQS 
themselves, because different subgroups $G$ and $H$
of $GL(N,\Bbbk)$ can give isomorphic quotients $\Bbbk^N/G$ and $\Bbbk^N/H$.
This happens, for example, if $G$ and $H$ are conjugate in $GL(N,\Bbbk)$.
At least for $\Bbbk=\C$ the converse statement also holds.
\begin{lemma}\label{L:conjgrps}
Let $G$ and $H$ be two finite subgroups of $GL(N,\C)$, $N\geq 2$. Assume
that $G$ and $H$ act without fixed points, so that $\C^N/G$ and $\C^N/H$
have isolated singularities. Then the quotients $\C^N/G$ and $\C^N/H$ are
isomorphic if and only if $G$ and $H$ are conjugate in $GL(N,\C)$.
\end{lemma}
\begin{proof}
Denote by $g$ and $h$ the canonical projections $g\colon
\C^N\to \C^N/G$ and $h\colon \C^N\to C^N/H$, by $0$ the origin in
$\C^N$, and let $P=g(0)$, $Q=h(0)$. The sufficiency condition
of the lemma is obvious, so let us prove necessity. 

Let $f\colon \C^N/G\to \C^N/H$ be an isomorphism. Since $G$ and $H$ act
freely on $\C^N\setminus\{0\}$, $g$ and $h$ restrict to
universal coverings $\C^N\setminus\{0\}\to (\C^N/G)\setminus\{P\}$ and
$\C^N\setminus\{0\}\to (\C^N/G)\setminus\{Q\}$. Clearly $f(P)=Q$, thus
$f$ lifts to an analytic isomorphism $\bar{f}\colon \C^N\setminus\{0\}\to
\C^N\setminus\{0\}$. But $f$ is also continuous at $P$, and it follows
that $\bar{f}$ extends continuously to $0$ by setting $\bar{f}(0)=0$.
This means that $0$ is a removable singular point of the analytic map
$\bar{f}$, and actually $\bar{f}$ is an analytic isomorphism in the
diagram
$$
\begin{xymatrix}{
&\C^N\ar[r]^{\bar{f}}\ar[d]_g &\C^N\ar[d]^h \\
&\C^N/G\ar[r]^f     &\C^N/H
}
\end{xymatrix}
$$
After fixing a reference point $O$ in $(\C^N/G)\setminus\{P\}$, $f$ induces 
an isomorphism
$$f_*\colon \pi_1((\C^N/G)\setminus\{P\},O)\to 
\pi_1((\C^N/H)\setminus\{Q\},f(O))$$
between fundamental groups, which in turn are isomorphic to $G$
and $H$ respectively. Fixing the lifting $\bar{f}$, we fix also an
isomorphism between $G$ and $H$ such that $\bar{f}$ becomes equivariant
with respect to it. Now observe that $G$ and $H$ also act naturally
on the tangent space $T_0\C^N$ which is canonically isomorphic to $\C^N$.
Let $A\colon T_0\C^N\to T_0\C^N$ be the differential of $\bar{f}$ at $0$.
Then
$$H=AGA^{-1}\,.$$
\end{proof}
\begin{lemma}\label{L:automorphism}
Let $\varphi,\psi\colon G\to GL(N,\Bbbk)$ be two exact linear 
representations of a finite group $G$ over a field $\Bbbk$. Then the images
$\varphi(G)$ and $\psi(G)$ are conjugate in $GL(N,\Bbbk)$ if and only if
there exists an automorphism $\alpha$ of the group $G$ such that
representations $\psi\circ\alpha$ and $\varphi$ are equivalent.
\end{lemma}
\begin{proof} The proof is straightforward and left to the reader. See also
\cite{Wolf}, Lemma~4.7.1.
\end{proof}
\begin{definition}
Let $p$ and $q$ be prime numbers, not necessarily distinct. We say that
a finite group $G$ satisfies the \emph{$pq$-condition}, if every subgroup
of $G$ of order $pq$ is cyclic.
\end{definition}
The following theorem is the cornerstone of the whole classification of
groups without fixed points. Note that it does not impose any conditions
on the characteristic of the base field $\Bbbk$.
\begin{theorem}\label{T:pq}
Let $G$ be a finite subgroup of $GL(N,\Bbbk)$ without fixed points, where
the field $\Bbbk$ is arbitrary. Then $G$ satisfies the $pq$-conditions
for all primes $p$ and $q$.
\end{theorem}
\begin{proof}
See \cite{Wolf}, Theorem~5.3.1.
\end{proof}

Now we can formulate a program which should give a classification of 
IQS $\Bbbk^N/G$ in the nonmodular case $\chr\Bbbk \nmid |G|$. 
This program was first suggested by Vincent \cite{Vincent} for $\Bbbk=\C$ 
as a method of solution of the Clifford-Klein Problem. First one classifies 
all finite groups satisfying all the $pq$-conditions. This is a purely 
group-theoretic problem completely settled for solvable groups by 
Zassenhaus, Vincent and Wolf. Next, for every such group $G$ one studies 
its linear representations over the given field $\Bbbk$, 
$\chr\Bbbk \nmid |G|$, and determines all exact irreducible
representations without fixed points. Arbitrary representation without
fixed points is a direct sum of irreducible ones. Finally, for every $G$
one calculates its group of automorphisms and determines when two 
irreducible representations without fixed points are equivalent modulo an 
automorphism. The last two steps were performed by Vincent and Wolf for 
$\Bbbk=\C$. It follows from Lemma~\ref{L:conjgrps} and other results of 
this section that in this way we indeed get a complete classification of 
isolated quotient singularities over $\C$ up to an analytic equivalence. 
For other fields some further identifications may be needed, i.~e., 
different output classes of the Vincent program might still give isomorphic 
quotient singularities. However, we do not have any examples of such 
phenomenon.

\subsection{Consequences of $pq$-conditions. Some representations}
In the particular case $p=q$, the $pq$-condition is called the 
\emph{$p^2$-condition}.
\begin{theorem}[\cite{Wolf}, Theorem~5.3.2]\label{T:Sylowsgrps}
If $G$ is a finite group, then the following conditions are equivalent:
\begin{enumerate}
\item $G$ satisfies the $p^2$-conditions for all primes $p$;
\item every Abelian subgroup of $G$ is cyclic;
\item if $p$ is an odd prime, then every Sylow $p$-subgroup of $G$ is
cyclic; Sylow $2$-subgroups of $G$ are either cyclic, or generalized
quaternion groups.
\end{enumerate}
\end{theorem}

We recall the definition of generalized quaternion groups below (see
\eqref{E:Qgroup}).

Groups $G$ such that all their Sylow $p$-subgroups are cyclic constitute
the simplest class of groups possessing representations without
fixed points (provided also that $G$ satisfies all the $pq$-conditions).
All such groups are solvable (\cite{Wolf}, Lemma~5.4.3) and even
metacyclic (\cite{Wolf}, Lemma~5.4.5). Nonsolvable groups without
fixed points necessarily contain Sylow $2$-subgroups isomorphic to
generalized quaternion groups. Classification of nonsolvable groups
without fixed points is available only over $\C$ and is based on the
following result of Zassenhaus.
\begin{theorem}[\cite{Wolf}, Theorem~6.2.1]
The binary icosahedral group $I^*$ is the only finite perfect group 
possessing representations without fixed points over $\C$.
\end{theorem}
The proof of this theorem is very technical and occupies about 15 pages
of \cite{Wolf} (the original proof of Zassenhaus was not complete).

Now let us describe some groups and their complex representations which
we use in the sequel.

\emph{Generalized quaternion group} $Q2^a$, $a\geq 3$, is a group of
order $2^a$ with generators $P$, $Q$, and relations
\begin{equation}\label{E:Qgroup}
Q^{2^{a-1}}=1\,,\; P^2=Q^{2^{a-2}}\,,\; PQP^{-1}=Q^{-1}\,.
\end{equation}
We get the usual quaternion group $Q8=\{\pm 1,\pm i,\pm j,\pm k\}$ for 
$a=3$.
\begin{lemma}[\cite{Wolf}, Lemma~5.6.2]\label{L:Qgrpreprs}
Let $Q2^a$, $a\geq 3$, be a generalized quaternion group and let $k$ be 
one of $2^{a-3}$ numbers $1$, $3$, $5,\dots,2^{a-2}-1$. Any irreducible
complex representation without fixed points of the group $Q2^a$ is 
equivalent to one of the following $2$-dimensional representations
$\alpha_k$:
$$\alpha_k(P)=
\begin{pmatrix}
   0 & 1 \\
   -1 & 0
\end{pmatrix},\;
\alpha_k(Q)=
\begin{pmatrix}
   e^{2\pi ik/2^{a-1}} & 0 \\
   0 & e^{-2\pi ik/2^{a-1}}
\end{pmatrix}.
$$
The representations $\alpha_k$ are pairwise nonequivalent.
\end{lemma}

\emph{Generalized binary tetrahedral group} $T_{v}^{*}$, $v\geq 1$,
has generators $X$, $P$, $Q$, and relations
$$X^{3^v}=P^4=1\,,\; P^2=Q^2\,,$$
$$XPX^{-1}=Q\,,\; XQX^{-1}=PQ\,,\; PQP^{-1}=Q^{-1}\,.$$
The usual binary tetrahedral group $T^*$ is $T_{1}^{*}$. $T_{v}^{*}$ has
order $8\cdot 3^v$.
\begin{lemma}[\cite{Wolf}, Lemma~7.1.3]\label{L:Tgrpreprs}
The binary tetrahedral group $T^*$ has only one irreducible complex
representation without fixed points. This is the standard representation
$\tau\colon T^*\to SU(2)$ obtained from the tetrahedral group $T<SO(3)$
and the double covering $SU(2)\to SO(3)$. If $v>1$, then $T_{v}^{*}$ has
exactly $2\cdot 3^{v-1}$ pairwise nonequivalent irreducible complex
representations without fixed points $\tau_k$, where $1\leq k< 3^v$,
$(k,3)=1$, given by
$$
\tau_k(X)=-\frac{1}{2} e^{2\pi ik/3^v}
\begin{pmatrix}
   1+i & 1+i \\
   -1+i & 1-i
\end{pmatrix},
$$
$$
\tau_k(P)=
\begin{pmatrix}
   i & 0 \\
   0 & -i
\end{pmatrix},\;
\tau_k(Q)=
\begin{pmatrix}
   0 & 1 \\
   -1 & 0
\end{pmatrix}.
$$
The representation $\tau$ of $T^*$ is also given by these formulae if
we formally set $v=0$.
\end{lemma}
\begin{proof}
The lemma is proved in \cite{Wolf}, only the precise matrices are not given
there. But their computation is an easy exercise which we omit.
\end{proof}

\emph{Generalized binary octahedral group} $O_{v}^{*}$, $v\geq 1$, has
generators $X$, $P$, $Q$, $R$, and relations
$$X^{3^v}=P^4=1\,,\; P^2=Q^2=R^2\,,$$
$$PQP^{-1}=Q^{-1}\,,\; XPX^{-1}=Q\,,\; XQX^{-1}=PQ\,,$$
$$RXR^{-1}=X^{-1}\,,\; RPR^{-1}=QP\,,\; RQR^{-1}=Q^{-1}\,.$$
The usual binary octahedral group $O^*$ is $O_{1}^{*}$. $O_{v}^{*}$ has
order $16\cdot 3^v$.
\begin{lemma}[\cite{Wolf}, Lemma~7.1.5]\label{L:Ogrpreprs}
The group $O^*$ has exactly $2$ irreducible complex representations
without fixed points, say $o_1$ and $o_2$. Images of $o_1$ and $o_2$ are
conjugate subgroups of $SU(2)$. For $o_1$ we may take the standard
representation obtained from the octahedral group $O<SO(3)$ and the double
covering $SU(2)\to SO(3)$. We may set $o_1(X)=\tau(X)$, $o_1(P)=\tau(P)$,
$o_1(Q)=\tau(Q)$ (see Lemma~\ref{L:Tgrpreprs}), and
$$o_1(R)=\frac{1}{\sqrt{2}}
\begin{pmatrix}
   1+i & 0 \\
   0 & 1-i
\end{pmatrix}.
$$
If $v>1$, then $O_{v}^{*}$ has exactly $3^{v-1}$
pairwise nonequivalent irreducible complex representations without
fixed points $o_k$, where $1\leq k<3^v$, $k\equiv 1(\mod 3)$, induced
by the representations $\tau_k$ (see Lemma~\ref{L:Tgrpreprs}) of the 
subgroup $T_{v}^{*}=\langle X,P,Q \rangle$, i.~e., $o_k$ has dimension 
$4$ and is given by
$$o_k(X)=
\begin{pmatrix}
   \tau_k(X) & 0 \\
   0 & \tau_k(X^{-1})
\end{pmatrix},\;
o_k(P)=
\begin{pmatrix}
   \tau_k(P) & 0 \\
   0 & \tau_k(QP)
\end{pmatrix},
$$
$$
o_k(Q)=
\begin{pmatrix}
   \tau_k(Q) & 0 \\
   0 & \tau_k(Q^{-1})
\end{pmatrix},\;
o_k(R)=
\begin{pmatrix}
   0 & 1 \\
   \tau_k(R^2) & 0
\end{pmatrix}.
$$
\end{lemma}

Recall that the \emph{binary icosahedral group} $I^*$ is obtained
from the group $I$ of rotations of icosahedron by the double
covering $SU(2)\to SO(3)$. For generators of $I$ we may take one of
the rotations by angle $2\pi/5$, which we denote by $V$, a rotation
$T$ of order $2$, whose axis has angle $\pi/3$ with the axis of $V$,
and another rotation $U$ of order $2$ whose axis is perpendicular to
aces of $V$ and $T$. We denote the corresponding generators of $I^*$ by
$\pm V$, $\pm T$, and $\pm U$.
\begin{lemma}[\cite{Wolf}, Lemma~7.1.7]\label{L:Igrpreprs}
The group $I^*$ has only $2$ irreducible complex representations without
fixed points $\iota_1$ and $\iota_{-1}$. They both have dimension $2$ and 
their images are conjugate subgroups of $SU(2)$. For $\iota_1$ we may take 
the standard representation
$$\iota_1(\pm V)=
\begin{pmatrix}
   \pm\varepsilon^3 & 0 \\
   0 & \pm\varepsilon^2
\end{pmatrix},
\iota_1(\pm U)=
\begin{pmatrix}
   0 & \mp 1 \\
   \pm 1 & 0
\end{pmatrix},
$$
$$\iota_1(\pm T)=\frac{1}{\sqrt{5}}
\begin{pmatrix}
   \mp(\varepsilon-\varepsilon^4) & \pm(\varepsilon^2-\varepsilon^3) \\
   \pm(\varepsilon^2-\varepsilon^3) & \pm(\varepsilon-\varepsilon^4)
\end{pmatrix},
$$
where $\varepsilon=e^{2\pi i/5}$.
\end{lemma}
\begin{proof}
The matrices of $\iota_1$ can be found in \cite{Klein}, Chapter II, \S 6.
\end{proof}

\section{Classification of isolated quotient singularities over $\C$}
\label{S:class}
This is the central section of our paper. Here we state the main theorems
on classification of IQS over the field $\C$ following \cite{Wolf}, 
Chapters~6 and 7.

As it was said in Section~\ref{S:prelim}, first we have to classify
finite groups with $pq$-conditions. According to \cite{Wolf}, a complete
abstract classification exists only for solvable groups. It is given below
in Theorem~\ref{T:solvgrps}. This is a purely group-theoretic result which
can be useful in classification of IQS over any field. 
Theorem~\ref{T:nonsolvgrps} classifies nonsolvable groups with 
$pq$-conditions possessing representations without fixed points over $\C$, 
and the last condition is essentially used. 
\begin{theorem}[\cite{Wolf}, Theorem~6.1.11]\label{T:solvgrps}
Let $G$ be a finite solvable group satisfying all the $pq$-conditions.
Then $G$ is isomorphic to one of the groups listed in 
Table~\ref{Tb:solvgrps}. Moreover, $G$ satisfies an additional condition: 
if $d$ is order of $r$ in the multiplicative group of residues modulo $m$ 
coprime with $m$, then every prime divisor of $d$ divides $n/d$ ($r$, $m$,
and $n$ are defined in Table~\ref{Tb:solvgrps}).
\end{theorem}
\begin{table}[hbt]
\begin{center}
   \begin{tabular}{|c|p{2.5cm}|p{3.2cm}|p{3.2cm}|c|}
   \hline
   Type & Generators & Relations & Conditions & Order \\ \hline
   I & $A$, $B$ & $A^m=B^n=1$, \newline $BAB^{-1}=A^r$ & $m\geq 1$, 
   $n\geq 1$, \newline $(n(r-1),m)=1$, \newline $r^n\equiv 1(m)$ & $mn$ 
   \\ \hline
   II & $A$, $B$, $R$ & As in I; also \newline 
   $R^2=B^{n/2}$, \newline $RAR^{-1}=A^l$, \newline $RBR^{-1}=B^k$ 
   & As in I; also \newline
   $l^2\equiv r^{k-1}\equiv 1(m)$, \newline $n=2^u v$, $u\geq 2$, \newline
   $k\equiv -1(2^u)$, \newline $k^2\equiv 1(n)$ & $2mn$ \\ \hline
   III & $A$, $B$, $P$, $Q$ & As in I; also \newline $P^4=1$, 
   \newline $P^2=Q^2=(PQ)^2$, \newline $AP=PA$, \newline $AQ=QA$, \newline 
   $BPB^{-1}=Q$, \newline $BQB^{-1}=PQ$
   & As in I; also \newline $n\equiv 1(2)$, \newline $n\equiv 0(3)$ & $8mn$ 
   \\ \hline
   IV & $A$, $B$, $P$, $Q$, $R$ & As in III; also \newline
   $R^2=P^2$, \newline $RPR^{-1}=QP$, \newline $RQR^{-1}=Q^{-1}$, \newline 
   $RAR^{-1}=A^l$, \newline $RBR^{-1}=B^k$
   & As in III; also \newline $k^2\equiv 1(n)$, \newline
   $k\equiv -1(3)$, \newline $r^{k-1}\equiv l^2\equiv 1(m)$ & $16mn$ 
   \\ \hline
   \end{tabular}
\end{center}
\caption{Solvable groups with $pq$-conditions}\label{Tb:solvgrps}
\end{table}
\begin{theorem}[\cite{Wolf}, Theorem~6.3.1]\label{T:nonsolvgrps}
Let $G$ be a nonsolvable finite group. If $G$ admits a representation
without fixed points over $\C$, then $G$ belongs to one of the following
two types.

Type V. $G=K\times I^*$, where $K$ is a solvable group of type I 
(see Theorem~\ref{T:solvgrps}), order of $K$ is coprime with $30$, and
$I^*$ is the binary icosahedral group.

Type VI. $G$ is generated by a normal subgroup $G_1$ of index $2$ and an 
element $S$, where $G_1=K\times I^*$ is of type V and $S$ satisfies
the following conditions. If we identify $I^*$ with the group $SL(2,5)$ of
$2\times 2$ matrices of determinant $1$ over the field $\Z/5$, then $S^2=
-I\in SL(2,5)$, and for any $L\in SL(2,5)$ $SLS^{-1}=\theta(L)$, where
$\theta$ ia an automorphism of $I^*$ given by
$$\theta(L)=
\begin{pmatrix}
   0 & -1 \\
   2 & 0
\end{pmatrix} L
{\begin{pmatrix}
   0 & -1 \\
   2 & 0
\end{pmatrix}}^{-1}\,.
$$
Moreover, any such $L$ belongs to the normalizer of the subgroup $K$. If
$A$ and $B$ are the generators of the group $K$, then $SBS^{-1}=B^k$,
$SAS^{-1}=A^l$, where integers $k$ and $l$ satisfy $l^2\equiv 1(m)$,
$k^2\equiv 1(n)$, $r^{k-1}\equiv 1(m)$ (see Table~\ref{Tb:solvgrps}).

Conversely, every group of type V or VI is nonsolvable and admits 
representations without fixed points over $\C$.
\end{theorem}
\begin{remark}\label{R:qgroup}
All groups listed in Theorems~\ref{T:solvgrps} and \ref{T:nonsolvgrps}
with the exception of groups of type I contain a generalized quaternion
group $Q2^a$ as a subgroup. This follows from Theorem~\ref{T:Sylowsgrps}
and Theorem~5.4.1 of \cite{Wolf}, the last stating that if all Sylow 
subgroups of a finite group $G$ are cyclic, then $G$ is a group of type I.
\end{remark}
\begin{remark}
The Kleinian subgroups of $SL(2,\C)$ are certainly among the groups
of types I -- VI. The cyclic group $\Z/n$ belongs to type I ($m=1$).
The binary dihedral group $D_{b}^{*}$ belongs to type I for odd $b$
($m=b$, $n=4$, $r=-1$), and to type II for even $b$ ($m=1$, $n=2b$,
$k=-1$). The binary tetrahedral group $T^*$ belongs to type III 
($m=1$, $n=3$). The binary octahedral group $O^*$ belongs to type IV
($m=1$, $n=3$, $k=-1$), and the binary icosahedral group to type V
($K=\{1\}$).
\end{remark}
\begin{remark}
It would be useful to write the action of the automprphism $\theta$ from
Theorem~\ref{T:nonsolvgrps} in terms of matrices $\iota_1(\pm V)$,
$\iota_1(\pm U)$, and $\iota_1(\pm T)$ of the standard representation
$\iota_1$ of the group $I^*$. Let us identify the image of this 
representation with $I^*$. Note that matrices $-V$ and $-T$ alone
generate $I^*$. We can fix an isomorphism with $SL(2,5)$ by setting,
for example,
$$-V \leftrightarrow
\begin{pmatrix}
-1 & -1 \\
0  & -1
\end{pmatrix},\:
-T \leftrightarrow
\begin{pmatrix}
2 & 0 \\
-1 & 3
\end{pmatrix}.
$$
Then $\theta$ is given by
$$\theta(-V)=\frac{1}{5}
\begin{pmatrix}
1-\varepsilon+2\varepsilon^2-2\varepsilon^4 & -2+2\varepsilon+
\varepsilon^2-\varepsilon^4 \\
2+\varepsilon-\varepsilon^3-2\varepsilon^4 & 1-2\varepsilon+
2\varepsilon^3-\varepsilon^4
\end{pmatrix},\:
\theta(-T)=
\begin{pmatrix}
0 & -\varepsilon \\
\varepsilon^4 & 0
\end{pmatrix},
$$
where $\varepsilon=e^{2\pi i/5}$.
\end{remark}

Next we need a classification of irreducible complex representations
without fixed points of groups from Theorems~\ref{T:solvgrps} and
\ref{T:nonsolvgrps}. The reference for this classification is \cite{Wolf},
Section~7.2.  Let us give a list of such representations for all types
I -- VI. The set of all irreducible complex representations without
fixed points of a given group $G$ is denoted by $\mathfrak{F}_{\C}(G)$,
$\varphi$ denotes here the Euler function. We also use the notation 
introduced in Theorems~\ref{T:solvgrps}, \ref{T:nonsolvgrps}, and 
Table~\ref{Tb:solvgrps}.

\begin{center}
\textbf{The list of irreducible complex representations \\
without fixed points}
\end{center}

\emph{Type I.} Let $G$ be a group of type I. Recall that $d$ is order of
$r$ in the group of residues modulo $m$ coprime with $m$. Then the set
$\mathfrak{F}_{\C}(G)$ consists of the following $\varphi(mn)/d^2$
representations $\pi_{k,l}$ of dimension $d$:
$$\pi_{k,l}(A)=
\begin{pmatrix}
   e^{2\pi ik/m} & 0 & \cdots & 0 \\
   0 & e^{2\pi ikr/m} & \cdots & 0 \\
   \vdots & \vdots & \ddots & \vdots \\
   0 & 0 & \cdots & e^{2\pi ikr^{d-1}/m}
\end{pmatrix},
$$
$$
\pi_{k,l}(B)=
\begin{pmatrix}
   0 & 1 & \cdots & 0 \\
   \vdots & \vdots & \ddots & \vdots \\
   0 & 0 & \cdots & 1 \\
   e^{2\pi il/n'} & 0 & \cdots & 0
\end{pmatrix}.
$$
Here $(k,m)=(l,n)=1$, $n'=n/d$.

\emph{Type II.} Let $G$ be a group of type II. Then it contains a subgroup
$\langle A,B\rangle$ of type I. All representations without fixed
points of the group $G$ are induced from the representations of the
subgroup $\langle A,B\rangle$. Namely, $\mathfrak{F}_{\C}(G)$ consists of
the following $\varphi(2mn)/(2d)^2$ representations $\alpha_{k',l'}$
of dimension $2d$:
$$\alpha_{k',l'}(A)=
\begin{pmatrix}
   \pi_{k',l'}(A) & 0 \\
   0 & \pi_{k',l'}(A^l)
\end{pmatrix},\;
\alpha_{k',l'}(B)=
\begin{pmatrix}
   \pi_{k',l'}(B) & 0 \\
   0 & \pi_{k',l'}(B^k)
\end{pmatrix},
$$
$$\alpha_{k',l'}(R)=
\begin{pmatrix}
   0 & I \\
   \pi_{k',l'}(B^{n/2}) & 0
\end{pmatrix}.
$$
Here $(k',m)=(l',n)=1$, $k$ and $l$ are defined in row II of 
Table~\ref{T:solvgrps}, and $I$ is the unit $d\times d$ matrix.

\emph{Type III.} Let $G$ be a group of type III. Then it contains a
subgroup $\langle A,B\rangle$ which is a group of type I of odd order.
We have to distinguish $3$ cases.

\textsc{Case 1.} $9\nmid n$, in particular, $3\nmid d$. In this case 
$$G=\langle A,B^3\rangle \times \langle B^{n''},P,Q\rangle\,,$$
where $n''=n/3$, $\langle B^{n''},P,Q\rangle$ is the binary tetrahedral
group $T^*$, and $\langle A,B^3\rangle$ is a group of type I with the same
value of $d$ as the group $\langle A,B\rangle$. Then $\mathfrak{F}_{\C}(G)$
consists of the following $\varphi(mn)/2d^2$ representations $\nu_{k,l}$
of dimension $2d$:
$$\nu_{k,l}=\pi_{k,l}\otimes\tau\,,$$
where $\pi_{k,l}\in \mathfrak{F}_{\C}(\langle A,B^3\rangle)$, and $\tau$
is the only irreducible representation without fixed points of the
group $T^*$, see Lemma~\ref{L:Tgrpreprs}.

\textsc{Case 2.} $9| n$, $3\nmid d$. Let $n=3^vn''$, where $(3,n'')=1$ and 
$v>1$. Then 
$$G=\langle A,B^{3^v}\rangle \times \langle B^{n''},P,Q\rangle=
\langle A,B^{3^v}\rangle \times T_{v}^{*}\,,$$
where $\langle A,B^{3^v}\rangle$ is a group of type I with the same value
of $d$ as $\langle A,B\rangle$, and $T_{v}^{*}$ is a generalized tetrahedral
group. Then $\mathfrak{F}_{\C}(G)$ consists of the following 
$\varphi(mn)/d^2$ representations $\nu_{k,l,j}$ of dimension $2d$:
$$\nu_{k,l,j}=\pi_{k,l}\otimes\tau_j\,,$$
where $\pi_{k,l}\in\mathfrak{F}_{\C}(\langle A,B^{3^v}\rangle)$, and
$\tau_j$ are defined in Lemma~\ref{L:Tgrpreprs}.

\textsc{Case 3.} $9|n$ and $3|d$. Then $G$ contains a normal subgroup of 
index $d$
$$\langle A,B^d,P,Q\rangle=\langle A\rangle\times\langle B^d\rangle\times
\langle P,Q\rangle\,,$$
where $\langle P,Q\rangle$ is the quaternion group $Q8$. Irreducible
representations without fixed points of $G$ are induced from irreducible
representations without fixed points of the subgroup $\langle A,B^d,P,Q
\rangle$. The set $\mathfrak{F}_{\C}(G)$ consists of the following
$\varphi(mn)/d^2$ representations $\mu_{k,l}$ of dimension $2d$:
$$\mu_{k,l}(A)=
\begin{pmatrix}
   e^{2\pi ik/m}I_2 & 0 & \cdots & 0 \\
   0 & e^{2\pi ikr/m}I_2 & \cdots & 0 \\
   \vdots & \vdots & \ddots & \vdots \\
   0 & 0 & \cdots & e^{2\pi ikr^{d-1}/m}I_2
\end{pmatrix},
$$
$$\mu_{k,l}(B)=
\begin{pmatrix}
   0 & I_{2d-2} \\
   e^{2\pi il/n'}I_2 & 0
\end{pmatrix},
$$
$$\mu_{k,l}(P)=
\begin{pmatrix}
   \alpha(P) & 0 & \cdots & 0 \\
   0 & \alpha(Q) & \cdots & 0 \\
   \vdots & \vdots & \ddots & \vdots \\
   0 & 0 & \cdots & \alpha(B^{d-1}PB^{1-d})
\end{pmatrix},
$$
$$\mu_{k,l}(Q)=
\begin{pmatrix}
   \alpha(Q) & 0 & \cdots & 0 \\
   0 & \alpha(PQ) & \cdots & 0 \\
   \vdots & \vdots & \ddots & \vdots \\
   0 & 0 & \cdots & \alpha(B^{d-1}QB^{1-d})
\end{pmatrix}.
$$
Here $(k,m)=(l,n)=1$, $n'=n/d$, $I_2$ and $I_{2d-2}$ are the unit matrices 
of dimensions $2\times 2$ and $(2d-2)\times(2d-2)$ respectively, and 
$\alpha$ is the representation of the quaternion group $Q8$ defined in 
Lemma~\ref{L:Qgrpreprs}.

\emph{Type IV.} Let $G$ be a group of type IV. Note that elements $A$,
$B$, $P$, $Q$ of $G$ generate a subgroup of type III. Here we again
have to consider several cases.

\textsc{Case 1.} $9\nmid n$, in particular $3\nmid d$.

\textsc{Subcase 1}a. Assume that there exists an element 
$\pi\in\mathfrak{F}(\langle A,B^3 \rangle)$ ($\langle A,B^3\rangle$ is a 
group of type I), which is equivalent to the representation $\pi'\colon 
g\mapsto \pi(RgR^{-1})$ of the group $\langle A,B^3\rangle$. Then $G$ is a 
direct product 
$$G=\langle A,B^3\rangle\times\langle R,B^{n/3},P,Q\rangle=
\langle A,B^3\rangle\times O^*\,,$$ 
where $O^*$ is the binary octahedral group. Then $\mathfrak{F}_{\C}(G)$
consists of the following $\varphi(mn)/d^2$ representations $\psi_{k,l,j}$
of dimension $2d$:
$$\psi_{k,l,j}=\pi_{k,l}\otimes o_j\,,$$
where $\pi_{k,l}\in\mathfrak{F}_{\C}(\langle A,B^3\rangle)$ and $o_j=o_1$ 
or $o_{-1}$ is a representation of the binary octahedral group $O^*$ 
defined in Lemma~\ref{L:Ogrpreprs}.

\textsc{Subcase 1}b. If the assumption of Subcase 1a is not satisfied, the 
group $G$ is not the direct product of subgroups $\langle A,B^3\rangle$ and
$O^*$. Then $\mathfrak{F}_{\C}(G)$ consists of $\varphi(mn)/4d^2$
representations $\gamma_{k',l'}$ of dimension $4d$ induced by 
representations $\nu_{k',l'}$ of the subgroup $\langle A,B,P,Q\rangle$. 
We have
$$\gamma_{k',l'}(A)=
\begin{pmatrix}
   \nu_{k',l'}(A) & 0 \\
   0 & \nu_{k',l'}(A^l)
\end{pmatrix},\;
\gamma_{k',l'}(B)=
\begin{pmatrix}
   \nu_{k',l'}(B) & 0 \\
   0 & \nu_{k',l'}(B^k)
\end{pmatrix},
$$
$$\gamma_{k',l'}(P)=
\begin{pmatrix}
   \nu_{k',l'}(P) & 0 \\
   0 & \nu_{k',l'}(QP)
\end{pmatrix},\;
\gamma_{k',l'}(Q)=
\begin{pmatrix}
   \nu_{k',l'}(Q) & 0 \\
   0 & \nu_{k',l'}(Q^{-1})
\end{pmatrix},
$$
$$\gamma_{k',l'}(R)=
\begin{pmatrix}
   0 & I_{2d} \\
   \nu_{k',l'}(P^2) & 0
\end{pmatrix}.
$$
Here $(k',m)=(l',n)=1$, $k$ and $l$ are defined in row IV of 
Table~\ref{T:solvgrps}, and $I_{2d}$ is the unit matrix of dimension
$2d\times 2d$.

\textsc{Case 2.} $9 | n$ but $3\nmid d$. Let $n=3^vn''$, $(3,n'')=1$. Then 
the subgroup $\langle R,B^{n''},P,Q\rangle$ of $G$ is isomorphic to the
generalized binary octahedral group $O_{v}^{*}$.

\textsc{Subcase 2}a. Assume that there exists an element 
$\pi\in\mathfrak{F}(\langle A,B^{3^v} \rangle)$ ($\langle A,B^{3^v}\rangle$ 
is a group of type I), which is equivalent to the representation 
$\pi'\colon g\mapsto \pi(RgR^{-1})$ of the group 
$\langle A,B^{3^v}\rangle$. Then $G$ is a direct product 
$$G=\langle A,B^{3^v}\rangle\times\langle R,B^{n''},P,Q\rangle=
\langle A,B^{3^v}\rangle\times O_{v}^{*}\,,$$
and the set $\mathfrak{F}_{\C}(G)$ consists of the following 
$\varphi(mn)/2d^2$ representations $\xi_{k,l,j}$ of dimension $4d$:
$$\xi_{k,l,j}=\pi_{k,l}\otimes o_j\,,$$
where $\pi_{k,l}\in\mathfrak{F}_{\C}(\langle A,B^{3^v}\rangle)$ and 
$o_j$ is a representation of the generalized binary octahedral group 
$O_{v}^{*}$ defined in Lemma~\ref{L:Ogrpreprs}.

\textsc{Subcase 2}b. If the assumption of Subcase 2a is not satisfied, the 
group $G$ is not the direct product of subgroups $\langle A,B^{3^v}\rangle$ 
and $O_{v}^{*}$. Then $\mathfrak{F}_{\C}(G)$ consists of $\varphi(mn)/2d^2$
representations $\gamma_{k',l',j}$ of dimension $4d$ induced by
representations $\nu_{k',l',j}$ of the subgroup $\langle A,B,P,Q\rangle$.
We have
$$\gamma_{k',l',j}(A)=
\begin{pmatrix}
   \nu_{k',l',j}(A) & 0 \\
   0 & \nu_{k',l',j}(A^l)
\end{pmatrix},
$$
$$\gamma_{k',l',j}(B)=
\begin{pmatrix}
   \nu_{k',l',j}(B) & 0 \\
   0 & \nu_{k',l',j}(B^k)
\end{pmatrix},
$$
$$\gamma_{k',l',j}(P)=
\begin{pmatrix}
   \nu_{k',l',j}(P) & 0 \\
   0 & \nu_{k',l',j}(QP)
\end{pmatrix},
$$
$$\gamma_{k',l',j}(Q)=
\begin{pmatrix}
   \nu_{k',l',j}(Q) & 0 \\
   0 & \nu_{k',l',j}(Q^{-1})
\end{pmatrix},\;
\gamma_{k',l',j}(R)=
\begin{pmatrix}
   0 & I_{2d} \\
   \nu_{k',l',j}(P^2) & 0
\end{pmatrix}.
$$
Here $(k',m)=(l',n)=1$, $j$ numerates representations of the group
$O_{v}^{*}$ (see Lemma~\ref{L:Ogrpreprs}), $k$ and $l$ are defined in 
row IV of Table~\ref{T:solvgrps}, and $I_{2d}$ is the unit matrix of 
dimension $2d\times 2d$.

\textsc{Case 3.} $3 | d$, in particular $9 | n$. The set 
$\mathfrak{F}_{\C}(G)$ consists of $\varphi(mn)/2d^2$ representations 
$\eta_{k',l'}$ of dimension $4d$ induced by representations $\mu_{k',l'}$ 
of the subgroup $\langle A,B,P,Q\rangle$. We have
$$\eta_{k',l'}(A)=
\begin{pmatrix}
   \mu_{k',l'}(A) & o \\
   0 & \mu_{k',l'}(A^l)
\end{pmatrix},\;
\eta_{k',l'}(B)=
\begin{pmatrix}
   \mu_{k',l'}(B) & 0 \\
   0 & \mu_{k',l'}(B^k)
\end{pmatrix},
$$
$$\eta_{k',l'}(P)=
\begin{pmatrix}
   \mu_{k',l'}(P) & 0 \\
   0 & \mu_{k',l'}(QP)
\end{pmatrix},\;
\eta_{k',l'}(Q)=
\begin{pmatrix}
   \mu_{k',l'}(Q) & 0 \\
   0 & \mu_{k',l'}(Q^{-1})
\end{pmatrix},
$$
$$\eta_{k',l'}(R)=
\begin{pmatrix}
   0 & I_{2d} \\
   \mu_{k',l'}(P^2) & 0
\end{pmatrix}.
$$
Here $(k',m)=(l',n)=1$, $k$ and $l$ are defined in row IV of 
Table~\ref{T:solvgrps}, and $I_{2d}$ is the unit matrix of dimension
$2d\times 2d$.

\emph{Type V.} Let $G$ be a group of type V. Then $G$ has the form
$G=K\times I^*$ and the set $\mathfrak{F}_{\C}(G)$ consists of
$2\varphi(mn)/d^2$ representations $\iota_{k,l,j}$ of dimension $2d$:
$$\iota_{k,l,j}=\pi_{k,l}\otimes \iota_j\,,$$
where $\pi_{k,l}$ are described in entry Type I of our List and 
$\iota_j=\iota_1$ or $\iota_{-1}$ are described in Lemma~\ref{L:Igrpreprs}.

\emph{Type VI.} Let $G$ be a group of type VI. Then the set 
$\mathfrak{F}_{\C}(G)$ consists of $\varphi(mn)/d^2$ representations
$\varkappa_{k',l',j}$ of dimension $4d$ induced by representations
$\iota_{k',l',j}$ of the subgroup $K\times I^*$. We have, in particular,
$$\varkappa_{k',l',j}(A)=
\begin{pmatrix}
   \iota_{k',l',j}(A) & 0 \\
   0 & \iota_{k',l',j}(A^l)
\end{pmatrix},\;
\varkappa_{k',l',j}(B)=
\begin{pmatrix}
   \iota_{k',l',j}(B) & 0 \\
   0 & \iota_{k',l',j}(B^k)
\end{pmatrix},
$$
$$\varkappa_{k',l',j}(S)=
\begin{pmatrix}
   0 & I_{2d} \\
   -I_{2d} & 0
\end{pmatrix},
$$
where $k$ and $l$ are defined in Theorem~\ref{T:nonsolvgrps}, and $I_{2d}$
is the unit matrix of dimension $2d\times 2d$. Matrices 
$\varkappa_{k',l',j}(\pm V)$, $\varkappa_{k',l',j}(\pm T)$,
$\varkappa_{k',l',j}(\pm U)$ for the generators of $I^*$ can be obtained
using Lemma~\ref{L:Igrpreprs} and Theorem~\ref{T:nonsolvgrps}.

\begin{theorem}\label{T:irreprs}
Let $G$ be a finite group possessing a complex representation without fixed
points. Then $G$ is one of the groups of types I -- VI and all irreducible 
complex representations of $G$ without fixed points are given in 
Table~\ref{Tb:reprs}. The columns of the table have the following meaning. 
The first column gives type of the group $G$, the second additional 
conditions on the group $G$, the third irreducible representations of $G$ 
without fixed points, the fourth the dimension of representations, the 
fifth column gives the determinant of the matrix of representation 
corresponding to the generator $B$ (sometimes some power of $B$) of the 
group $G$, and the sixth the conditions when the image of a representation 
is contained in $SL(d,\C)$ ($SL(2d,\C)$, $SL(4d,\C)$). All the other 
generators of the group $G$ have determinant $1$. We use in the table the 
notation introduced above in Theorems~\ref{T:solvgrps}, 
\ref{T:nonsolvgrps}, and the List of irreducible representations without 
fixed points. $D_{n}^{*}$ denotes the binary dihedral group of order $4n$.
\end{theorem}
\begin{table}[hbt]
\begin{center}
   \begin{tabular}{|c|p{3cm}|c|c|p{2.5cm}|p{2.3cm}|}
   \hline
   Type & Case & Repr. & $\dim$ & $\det(B)$ 
   & $G<SL$ \\ \hline
   I & \phantom{0} & $\pi_{k,l}$ & $d$ & $(-1)^{d-1}e^{2\pi il/n'}$ 
   & $n=2^{s+1}$, \newline $d=2^s$, $s\geq 1$, \newline
   or $d=1$ and $G=\{1\}$ \\ \hline
   II & \phantom{0} & $\alpha_{k',l'}$ & $2d$ & $e^{2\pi il'(k+1)/n'}$
   & $d=1$ and \newline $G=D_{n}^{*}$, \newline or $d=2$ and $k=-1$ 
   \\ \hline
   III & $9\nmid n$ & $\nu_{k,l}$ & $2d$ & $|\nu_{k,l}(B^3)|= \newline
   e^{12\pi il/n'}$ & $d=1$ and \newline $G=T^*$ \\ \hline
   III & $9|n$, $3\nmid d$ & $\nu_{k,l,j}$ & $2d$ & $|\nu_{k,l,j}(B^{n''})|=
   \newline e^{4\pi ijd/3^v}$, $|\nu_{k,l,j}(B^{3^v})|=\newline 
   e^{4\pi il3^v/n'}$ & never \\ \hline
   III & $3|d$ & $\mu_{k,l}$ & $2d$ & $e^{4\pi il/n'}$ & never \\ \hline
   IV & $9\nmid n$, \newline $G=\langle A,B^3\rangle\times O^*$ 
   & $\psi_{k,l,j}$ & $2d$ & $|\psi_{k,l,j}(B^3)|= \newline 
   e^{4\pi il/(n'/3)}$ & $d=1$ and \newline $G=O^*$ \\ \hline
   IV & $9\nmid n$, \newline $G\ne\langle A,B^3\rangle\times O^*$ 
   & $\gamma_{k',l'}$ & $4d$ & $|\gamma_{k',l'}(B^3)|= \newline
   e^{4\pi il'(k+1)/(n'/3)}$ & $d=1$, $A=1$, and $k=-1$ \\ \hline
   IV & $9|n$, $3\nmid d$, \newline $G=\langle A,B^{3^v}\rangle\times 
   O_{v}^{*}$ & $\xi_{k,l,j}$ & $4d$ & $|\xi_{k,l,j}(B^{3^v})|= \newline
   e^{2\pi il/(n''/d)}$ & $d=1$ and \newline $G=O_{v}^{*}$ \\ \hline
   IV & $9|n$, $3\nmid d$, \newline $G\ne\langle A,B^{3^v}\rangle\times 
   O_{v}^{*}$ & $\gamma_{k',l',j}$ & $4d$ & $|\gamma_{k',l',j}(B^{3^v})|=
   \newline e^{2\pi il'(k+1)/(n''/d)}$ & $d=1$, $A=1$ \\ \hline
   IV & $3|d$ & $\eta_{k',l'}$ & $4d$ & $e^{4\pi il'(k+1)/n'}$
   & never \\ \hline
   V & \phantom{0} & $\iota_{k,l,j}$ & $2d$ & $e^{4\pi il/n'}$ 
   & $d=1$ and \newline $G=I^*$ \\ \hline
   VI & \phantom{0} & $\varkappa_{k',l',j}$ & $4d$ & $e^{4\pi il'(k+1)/n'}$
   & $d=1$, $A=1$, $k=-1$ \\
   \hline
   \end{tabular}
\end{center}
\caption{Irreducible representations without fixed points}\label{Tb:reprs}
\end{table}
\begin{proof}
Everything with the exception of computation of determinants is done
in \cite{Wolf}, Section~7.2. So let us concentrate on columns 5 and 6
of Table~\ref{Tb:reprs}.

\emph{Type I.} From the matrices given in the corresponding entry of
the List of irreducible representations without fixed points one easily
finds 
$$|\pi_{k,l}(A)|=\exp\left(\frac{2\pi ik}{m}(1+\dots+r^{d-1})\right)=1\,,$$
because $m|(r^d-1)$ but $(r-1,m)=1$, see Table~\ref{Tb:solvgrps}. Further,
$$|\pi_{k,l}(B)|=(-1)^{d-1}e^{2\pi il/n'}\,.$$
Since $(l,n')=1$, $e^{2\pi il/n'}$ is a primitive root of unity of degree
$n'$. It follows that the last determinant can be equal to $1$ either when 
$d=n'=n=m=1$, that is when the group $G$ is trivial, or when $n'=2$. Since 
any prime divisor of $d$ divides $n'$, we have $d=2^s$, $n=n'd=2^{s+1}$, 
$s\geq 1$. The case $s=0$, i.~e. $d=1$, again leads to the trivial group 
$G$. We have proved the theorem for groups of type I.

\emph{Type II.} From the matrices of the representation $\alpha_{k',l'}$
and from row I of Table~\ref{Tb:reprs} which we just proved one easily
finds 
$$|\alpha_{k',l'}(A)|=|\pi_{k',l'}(A)||\pi_{k',l'}(A^l)|=1\,,$$
$$|\alpha_{k',l'}(R)|=(-1)^d|\pi_{k',l'}(B^{n/2})|=(-1)^d e^{\pi il'd}=1,$$
and
$$|\alpha_{k',l'}(B)|=|\pi_{k',l'}(B)||\pi_{k',l'}(B^k)|=
e^{2\pi i(k+1)l'/n'},$$
where one has to use that $n$ and $k+1$ are divisible by $4$. If
determinant of $B$ is equal to $1$, then $n'|k+1$, and thus any prime
divisor of $d$ divides also $k+1$. On the other hand, $d$ divides
$k-1$ (see row II of Table~\ref{Tb:solvgrps}). It follows that $d=1$ or
$d=2$. In the first case we get $m=1$ and $A=1$, thus 
$$G=\langle B,R\,|\,R^2=B^{n/2},\; RBR^{-1}=B^{-1},\; B^n=1\rangle$$
is the binary dihedral group and $\alpha_{k',l'}$ is its standard 
representation. In the second case we get $n'=2^{u-1}v$. On the other hand,
$(k+1)/n'$ must be integer, but $k$ is defined modulo $n$. Thus
we may take $k=-1$ and the group reduces to $G=\langle A,B,R \,|\,
A^m=B^n=1,\; R^2=B^{n/2},\; BAB^{-1}=A^r,\; RAR^{-1}=A^l,\; 
RBR^{-1}=B^{-1}\rangle$.

\emph{Type III.} Recall that if $S$ and $T$ are matrices of size
$m\times m$ and $n\times n$ respectively, then
$$|S\otimes T|=|A|^n |B|^m\,.$$

\textsc{Case 1.} $9\nmid n$. From the definition of the representation
$\nu_{k,l}$ we get
$$|\nu_{k,l}(A)|=|\pi_{k,l}\otimes\tau(A)|=1\,,$$
$$|\nu_{k,l}(B^3)|=|\pi_{k,l}(B^3)|^2=e^{12\pi il/n'}\,.$$
Matrices of $B^{n/3}$, $P$, and $Q$ all have determinant $1$ since
they generate the binary tetrahedral group. Recall that now $n'$ is odd,
thus $|\nu_{k,l}(B^3)|=1$ only if $n'=1$ or $n'=3$. But the first case
is impossible, because it would imply that $n=1$, but $n$ is divisible by 
$3$. In the second case we have $d=1$ or $d=3$. Again the second case
is impossible because it would imply $9|n$. We conclude that $d=1$,
$G=T^*$, and $\nu_{k,l}=\tau$.

\textsc{Case 2.} $9|n$, but $3\nmid d$. In this case the group $G$
is the direct product of the subgroups $\langle A,B^{3^v}\rangle$ and
$\langle B^{n''},P,Q\rangle\simeq T_{v}^{*}$, where $n=3^vn''$, $v\geq 2$,
and $B^{n''}$ corresponds to the generator $X$ of the group $T_{v}^{*}$.
It is easy to check that $|\tau_j(X)|=e^{4\pi ij/3^v}$ (see 
Lemma~\ref{L:Tgrpreprs}), which is never
equal to $1$. It follows that $|\nu_{k,l,j}(B^{n''})|=e^{4\pi ijd/3^v}$
and the image of the group $G$ is never contained in the group
$SL(2d,\C)$. It is also a straightforward verification that
$|\nu_{k,l,j}(A)|=|\nu_{k,l,j}(P)|=|\nu_{k,l,j}(Q)|=1$ and 
$|\nu_{k,l,j}(B^{3^v})|=e^{4\pi il3^v/n'}$.

\textsc{Case 3.} $3|n$. From the definition of the representation
$\mu_{k,l}$ one finds that determinants of the matrices corresponding
to $A$, $P$, and $Q$ are all equal to $1$, whereas $|\mu_{k,l}(B)|=
e^{4\pi il/n'}$. Note that $n'$ is divisible by $3$, $(n',l)=1$, thus the 
last determinant is never equal to $1$ and the image of $\mu_{k,l}$ is 
never contained in $SL(2d,\C)$.

\emph{Type IV.} 

\textsc{Case 1}a. $9\nmid n$, $G$ is the direct product
$\langle A,B^3\rangle\times O^*$, where $O^*=\langle B^{n/3},P,Q,R\rangle$.
From the definition of the representation $\psi_{k,l,j}$ one finds
$$|\psi_{k,l,j}(A)|=|\psi_{k,l,j}(B^{n/3})|=|\psi_{k,l,j}(P)|=
|\psi_{k,l,j}(Q)|=|\psi_{k,l,j}(R)|=1\,,$$
$$|\psi_{k,l,j}(B^3)|=|\pi_{k,l}\otimes o_j(B^3)|=e^{4\pi il/(n'/3)}\,.$$
Note that $n'$ is an odd number, so the only possibility for the last
determinant to be equal to $1$ is $d=1$, $n=n'=3$, thus $G=O^*$.

\textsc{Case 1}b. Conditions are as in Case 1a, but $G$ is not the direct
product. From the definition of the representation $\gamma_{k',l'}$ it
easily follows that determinants of matrices corresponding to $A$,
$B^{n/3}$, $P$, $Q$, $R$ are all equal to $1$. For $B^3$ we have
$$|\gamma_{k',l'}(B^3)|=|\nu_{k',l'}(B^3)||\nu_{k',l'}(B^{3k})|=
e^{4\pi il'(k+1)/(n'/3)}\,.$$
By an argument analogous to that of Type II we prove that either $d=2$
or $d=1$. But here $n$ is odd, thus $d=1$. This implies $A=1$. Further,
$n/3$ divides $k+1$, and at the same time $3$ divides $k+1$. Since
$k$ is determined modulo $n$, we may take $k=-1$.

\textsc{Case 2}a. $9|n$, $3\nmid d$, and $G$ is the direct product
$\langle A,B^{3^v}\rangle\times O_{v}^{*}$, where $O_{v}^{*}$ is
generated by $B^{n''}$, $P$, $Q$, and $R$, $n=3^vn''$, $(3,n'')=1$. First
please check that $|o_j(X)|=1$, where $X$ is a generator of the
generalized binary octahedral group. This follows easily from 
Lemma~\ref{L:Ogrpreprs}. Then it follows from the definition of
the representation $\xi_{k,l}$ that the matrices corresponding
to $A$, $B^{n''}$, $P$, $Q$, $R$ all have determinant $1$. Further,
$$|\xi_{k,l}(B^{3^v})|=e^{8\pi il/(n''/d)}\,.$$
This determinant equals $1$ only if $n''=d$. But any prime divisor of
$d$ divides $n''/d$, hence $n''=d=1$, and $G=O_{v}^{*}$.

\textsc{Case 2}b. The conditions are as in Case 2a, but $G$ is not the
direct product. From the definition of the representation 
$\gamma_{k',l',j}$ one easily computes that all determinants are equal
to $1$ with the exception of
$$|\gamma_{k',l',j}(B^{3^v})|=e^{2\pi il'(k+1)/(n''/d)}\,.$$
If this determinant equals $1$ and $p$ is a prime divisor of $d$, we
again see that $p$ must divide both $k-1$ and $k+1$, thus $p=2$. But
in this case $n$ is odd, so the only possibility is $d=1$ and $A=1$.

\textsc{Case 3.} $3|d$. From the definition of the representation
$\eta_{k',l'}$ one directly finds that all the determinants involved
equal $1$, with the exception of
$$|\eta_{k',l'}(B)|=e^{4\pi il'(k+1)/n'}.$$
This determinant can be equal to $1$ only if $d=1$, but this would
contradict $3|d$. Therefore $\eta_{k',l'}(G)$ is never contained in
$SL(4d,\C)$ in this case.

\emph{Type V.} Now $G$ a direct product $K\times I^*$, where $K$ is a
group of type I and order of $G$ is coprime with $30$. It follows from the 
definition of the representation $\iota_{k,l,j}$ and Lemma~\ref{L:Igrpreprs}
that matrices corresponding to generators of $G$ all have determinant $1$
with the exception of $|\iota_{k,l,j}(B)|$ which equals $e^{4\pi il/n'}$.
The number $n'$ is odd, thus this determinant can be equal to $1$ only
if $n'=1$. This implies $d=n=1$ and $K$ is trivial. Hence $G=I^*$ is the
binary icosahedral group and $\iota_{k,l,j}$ is one of its representations
described in Lemma~\ref{L:Igrpreprs}.

\emph{Type VI.} Again an easy computation that we omit shows that the 
question can be only in the determinant $|\varkappa_{k',l',j}(B)|$.
From the definition of $\varkappa_{k',l',j}$ one finds
$$|\varkappa_{k',l',j}(B)|=e^{4\pi il'(k+1)}\,.$$
The same type of argument as for Type II shows that this can be equal to 
$1$ only if $d=1$, $A=1$, and $k=-1$.
\end{proof}

\begin{remark}
Note that if one of the irreducible representations without fixed points
of a group $G$ is contained in the special linear group, then the others 
are also. Thus this seems to be a property of the group $G$ rather than
of a particular representation.
\end{remark}
\begin{remark}
For a given group $G$ from Theorems~\ref{T:solvgrps} and 
\ref{T:nonsolvgrps}, all its irreducible representations have the same
dimension $d$, $2d$, or $4d$. Thus any representation without fixed
points of $G$ has dimension multiple to $d$, $2d$, or $4d$ respectively.
\end{remark}
\begin{remark}
Note also that all representations from Theorem~\ref{T:irreprs} are 
imprimitive and induced from primitive representations of dimension $1$
or $2$.
\end{remark}

Finally we have to determine the automorphisms of groups of types I--VI
and the action of the automorphisms on the irreducible representations
of these groups by the rule $\varphi\to\varphi\circ\alpha$, where $\varphi$ 
is a representation and $\alpha$ is an automorphism. This will give the 
conditions when the images of two representations are conjugate in $GL$.
\begin{theorem}[\cite{Wolf}, Theorem~7.3.21]\label{T:aut}
Let $G$ be a finite group possessing a representation without fixed points,
i.~e., one of the groups listed in Theorems~\ref{T:solvgrps} and
\ref{T:nonsolvgrps}. Then the action of the automorphisms on the
irreducible representations of $G$ is described in Table~\ref{Tb:aut}. We
use here the notation introduced in Table~\ref{Tb:solvgrps}, 
Theorem~\ref{T:nonsolvgrps}, and the List of irreducible complex 
representations without fixed points.
\end{theorem}
\begin{table}[hbt]
\begin{center}
   \begin{tabular}{|c|p{3cm}|c|p{2.8cm}|}
   \hline
   Type & Case & Action & Conditions  \\ \hline
   I & \phantom{0} & $A_{a,b}\colon \pi_{k,l}\mapsto \pi_{ak,bl}$
   & $(a,m)=1$, \newline $(b,n)=1$, \newline $b\equiv 1(d)$ \\ \hline
   II & \phantom{0} & $A_{a,b}\colon \alpha_{k',l'}\mapsto 
   \alpha_{ak',bl'}$ & $(a,m)=1$, \newline $(b,n)=1$, 
   \newline $b\equiv 1(d)$  \\ \hline
   III & $9\nmid n$ & $A_{a,b}\colon \nu_{k,l}\mapsto \nu_{ak,bl}$ &  
   $(a,m)=1$, \newline $(b,n/3)=1$, \newline $b\equiv 1(d)$ \\ \hline
   III & $9|n$, $3\nmid d$ & $A_{a,b,c}\colon\nu_{k,l,j}\mapsto
   \nu_{ak,bl,cj}$ & $(a,m)=1$, \newline $(b,n/3^v)=1$, \newline
   $(c,3)=1$, \newline $b\equiv 1(d)$ \\ \hline
   III & $3|d$ & $A_{a,b}\colon \mu_{k,l}\mapsto \mu_{ak,bl}$ 
   & $(a,m)=1$, \newline $(b,n)=1$, \newline $b\equiv 1(d)$ \\ \hline
   IV & $9\nmid n$, \newline $G=\langle A,B^3\rangle\times O^*$ 
   & $A_{a,b,c}\colon \psi_{k,l,j}\mapsto \psi_{ak,bl,cj}$ 
   & $(a,m)=1$, \newline $(b,n/3)=1$, \newline $b\equiv 1(d)$,
   $c=\pm 1$ \\ \hline
   IV & $9\nmid n$, \newline $G\ne\langle A,B^3\rangle\times O^*$ 
   & $A_{a,b}\colon \gamma_{k',l'}\mapsto \gamma_{ak',bl'}$ 
   & $(a,m)=1$, \newline $(b,n/3)=1$, \newline $b\equiv 1(d)$ \\ \hline
   IV & $9|n$, $3\nmid d$, \newline $G=\langle A,B^{3^v}\rangle\times 
   O_{v}^{*}$ & $A_{a,b,c}\colon \xi_{k,l,j}\mapsto \xi_{ak,bl,cj}$ 
   & $(a,m)=1$, \newline $(b,n/3^v)=1$, \newline $(c,3)=1$, \newline 
   $b\equiv 1(d)$ \\ \hline
   IV & $9|n$, $3\nmid d$, \newline $G\ne\langle A,B^{3^v}\rangle\times 
   O_{v}^{*}$ & $A_{a,b,c}\colon \gamma_{k',l',j}\mapsto 
   \gamma_{ak',bl',cj}$ & $(a,m)=1$, \newline $(b,n/3^v)=1$, \newline
   $(c,3)=1$, \newline $b\equiv 1(d)$ \\ \hline
   IV & $3|d$ & $A_{a,b}\colon \eta_{k',l'}\mapsto \eta_{ak',bl'}$ 
   & $(a,m)=1$, \newline $(b,n)=1$, \newline $b\equiv 1(d)$ \\ \hline
   V & \phantom{0} & $A_{a,b,c}\colon \iota_{k,l,j}\mapsto 
   \iota_{ak,bl,cj}$ & $(a,m)=1$, \newline $(b,n)=1$, \newline 
   $b\equiv 1(d)$, $c=\pm 1$ \\ \hline
   VI & \phantom{0} & $A_{a,b,c}\colon \varkappa_{k',l',j}\mapsto
   \varkappa_{ak',bl',cj}$ & $(a,m)=1$, \newline $(b,n)=1$, \newline 
   $b\equiv 1(d)$, $c=\pm 1$ \\
   \hline
   \end{tabular}
\end{center}
\caption{Action of automorphisms on representations}\label{Tb:aut}
\end{table}

Now we are able to formulate the main theorem on classification of IQS
over the field $\C$.
\begin{theorem}\label{T:classisqsing}
Let $Q\in X$ be an isolated quotient singularity of a variety $X$ defined
over the field $\C$, as it is described in Section~\ref{S:prelim}. Then
locally analytically at $Q$ the variety $X$ is isomorphic to the
quotient $\C^N/\varphi(G)$, where $N=\dim X$, $G$ is one of the groups 
described in Theorems~\ref{T:solvgrps} and \ref{T:nonsolvgrps}, and 
$\varphi$ is a direct sum of irreducible representations of $G$ described 
in Theorem~\ref{T:irreprs}. 

Conversely, any group from Theorems~\ref{T:solvgrps} and 
\ref{T:nonsolvgrps} acting on $\C^N$ via a direct sum $\varphi$ of 
irreducible representations described in Theorem~\ref{T:irreprs} gives an 
isolated singularity $O\in\C^N/\varphi(G)$. If $\varphi$ and $\psi$ are two 
such representations, then singularities $\C^N/\varphi(G)$ and 
$\C^N/\psi(G)$ are isomorphic if and only if $\varphi$ and $\psi$ can be 
transformed into each other by an action of an automorphism of the group 
$G$ as described in Theorem~\ref{T:aut}.
\end{theorem}

\begin{remark}
The conditions for singularities $\C^N/\varphi(G)$ and $\C^N/\psi(G)$ to
be isomorphic can be stated more explicitely. An interested reader can
consult \cite{Wolf}, Sections~7.3 and 7.4.
\end{remark}

This classification generalizes almost literally to other algebraically 
closed fields of characteristic $0$.
\begin{theorem}
Let $\Bbbk$ be an algebraically closed field of characteristic $0$. Let
$Q\in X$ be an isolated singularity of an algebraic variety $X$ defined
over $\Bbbk$. Then at the point $Q$ the variety $X$ is formally isomorphic
to $\Bbbk^N/\varphi(G)$, where $N=\dim X$, $G$ is one of the groups 
described in Theorems~\ref{T:solvgrps} and \ref{T:nonsolvgrps}, and 
$\varphi$ is a direct sum of irreducible representations without fixed 
points of $G$ described in Theorem~\ref{T:irreprs}, where in matrices of 
representations one has to replace the complex roots of unity with 
roots of unity of corresponding degree contained in the field $\Bbbk$.
\end{theorem}
\begin{proof}
First let us recall some terminology. Let $\varphi\colon G\to GL(V_\Bbbk)$
be a representation of a group $G$ on a finitely dimensional $\Bbbk$-vector
space $V_\Bbbk$. We say that $\varphi$ is defined over a subfield
$L\subseteq\Bbbk$, if $\varphi$ is obtained from a representation 
$\varphi'\colon G\to GL(V_L)$ over $L$ by extension of scalars: $V_\Bbbk=
V_L\otimes\Bbbk$.

Since $\Bbbk$ is algebraically closed field of characteristic $0$, we
may assume that it contains the field $\overline{\Q}$ of algebraic numbers.
By Brower's Theorem (\cite{Serr}, 12.2) any representation of a finite
group $G$ over $\Bbbk$ is actually defined over $\overline{\Q}$ (in fact
over a smaller field). It follows that everything that we need about
classification of groups without fixed points and their linear 
representations and that we have proved over $\C$ holds also over $\Bbbk$.
\end{proof}

We can not yet tell much on IQS over fields of positive characteristic.
The only result we have is the following.
\begin{theorem}
Let $\Bbbk$ be an algebraically closed field of characteristic $0$ or
$p>0$. Let $G$ be a finite group of order not divisible by $p$, $N$
an odd number, and assume that $G$ acts linearly without quasireflections on 
the affine space $\Bbbk^N$. Then $\Bbbk^N/G$ is an isolated singularity
if and only if $G$ is a metacyclic group satisfying the conditions of
Theorem~\ref{T:solvgrps}, type I.
\end{theorem}
\begin{proof}
It follows from the conditions of the theorem that $G$ acts on $\Bbbk^N$
without fixed points, and thus its Sylow subgroups are either cyclic
or generalized quaternion (Theorem~\ref{T:Sylowsgrps}). But every exact
$\Bbbk$-linear irreducible representation of a generalized quaternion 
group is $2$-dimensional (\cite{CR}, \S 47). If $Q2^a<G$ is such a subgroup,
then $\varphi$ restricted to $Q2^a$ is an exact representation, but then
$N$ must be even. This shows that every Sylow subgroup of $G$ is cyclic.
Such groups are classified in Theorem~\ref{T:solvgrps}, type I (see
also \cite{Wolf}, Theorem~5.4.1), all of them are metacyclic.
\end{proof}
\begin{remark}
Not all irreducible representations of metacyclic groups have the form
given in our List above, type I (where one has to take roots of unity of 
the field $\Bbbk$ instead of complex roots), see \cite{CR}, \S 47. An 
example of a group of type I satisfying all the $pq$-conditions but 
admitting irreducible representations different from type I of the List
is given by $m=7\cdot 19$, $n=27$, $r=4$, so that $d=9$.
\end{remark}

On the other hand, representations of metacyclic groups over $\Bbbk$,
$\chr\Bbbk\nmid |G|$, have been completely described in the literature
(\cite{Tucker}), and it should be possible to extract a classification
of representations without fixed points from there.

\section{Gorenstein isolated quotient singularities}\label{S:Gorenstein}
In this section we work over the field $\C$ of complex numbers. First
let us deduce Theorem~\ref{T:KN} of Kurano and Nishi from the classification
of IQS over the field $\C$. 

\emph{Proof of Theorem~\ref{T:KN}.} Let $p$ be an odd prime, and let
$G$ be a finite group acting on $\C^p$ via a representation $\varphi\colon
G\to GL(p,\C)$. We may assume that $G$ acts without quasireflections.
If $\C^p/\varphi(G)$ is a Gorenstein isolated singularity, by 
Theorem~\ref{T:classisqsing} $G$ is one of the groups of types I -- VI,
$\varphi$ is a direct sum of irreducible representations of dimension $d$,
$2d$, or $4d$ from Theorem~\ref{T:irreprs}, and $\varphi(G)\subset 
SL(p,\C)$ by Theorem~\ref{T:Watanabe}. Since $p$ is odd, $G$ has type I. 
Thus either $d=1$ and $G$ is cyclic, or $d=p$. But it follows from
Theorem~\ref{T:irreprs}, type I, that the last case is impossible.
\begin{flushright}{$\Box$}\end{flushright}

If we want to classify all Gorenstein IQS over the field $\C$, we have
only to refine the classification of Theorem~\ref{T:classisqsing} by
taking into account the determinants of representations.
\begin{theorem}\label{T:classgiqsing}
Let $X=\C^N/G'$ be a Gorenstein isolated quotient singularity, where
$G'$ is a finite subgroup of $\C^N$. Then the variety $X$ is isomorphic
to the quotient $\C^N/\varphi(G)$, where $G$ is one of the groups
described in Theorem~\ref{T:solvgrps} and Theorem~\ref{T:nonsolvgrps},
and $\varphi$ is a direct sum $\varphi=\varphi_1\oplus\dots\oplus
\varphi_s$, $s\geq 1$, of irreducible representations described in 
Theorem~\ref{T:irreprs} and satisfying the additional condition
$$\det(\varphi_1(B))\cdot\det(\varphi_2(B))\cdots\det(\varphi_s(B))=1\,,$$
where the values $\det(\varphi_i(B))$, $1\leq i\leq s$, are given in the 
5th column of Table~\ref{Tb:reprs}. More explicitely this condition is 
stated in Table~\ref{Tb:gorcond}, where we use the same notation as in 
Table~\ref{Tb:reprs}. The representations $\varphi_i$ have one 
and the same dimension $d$, $2d$, or $4d$, $d|N$ ($2d|N$, $4d|N$). 
\end{theorem}
\begin{table}[hbt]
\begin{center}
   \begin{tabular}{|c|p{3cm}|c|p{4.8cm}|}
   \hline
   Type & Case & Representation & Condition  \\ \hline
   I & \phantom{0} & $\varphi_i=\pi_{k_i,l_i}$ 
   & $2\sum\limits_{i=1}^{s} l_i\equiv sn'(d-1)(\mod 2n')$ \\ \hline
   II & \phantom{0} & $\varphi_i=\alpha_{k_i',l_i'}$ 
   & $(k+1)\sum\limits_{i=1}^{s} l_i'\equiv 0(\mod n')$ \\ \hline
   III & $9\nmid n$ & $\varphi_i=\nu_{k_i,l_i}$ 
   & $\sum\limits_{i=1}^{s} l_i\equiv 0(\mod n'/3)$ \\ \hline
   III & $9|n$, $3\nmid d$ & $\varphi_i=\nu_{k_i,l_i,j_i}$ 
   & $\sum\limits_{i=1}^{s} j_i\equiv 0(\mod 3^v)$, \newline
   $\sum\limits_{i=1}^{s} l_i\equiv 0(\mod n'/3^v)$ \\ \hline
   III & $3|d$ & $\varphi_i=\mu_{k_i,l_i}$ 
   & $\sum\limits_{i=1}^{s} l_i\equiv 0(\mod n')$ \\ \hline
   IV & $9\nmid n$, \newline $G=\langle A,B^3\rangle\times O^*$ 
   & $\varphi_i=\psi_{k_i,l_i,j_i}$ 
   & $\sum\limits_{i=1}^{s} l_i \equiv 0(\mod n'/3)$ \\ \hline
   IV & $9\nmid n$, \newline $G\ne\langle A,B^3\rangle\times O^*$ 
   & $\varphi_i=\gamma_{k_i',l_i'}$ 
   & $(k+1)\sum\limits_{i=1}^{s} l_i' \equiv 0(\mod n'/3)$ \\ \hline
   IV & $9|n$, $3\nmid d$, \newline $G=\langle A,B^{3^v}\rangle\times 
   O_{v}^{*}$ & $\varphi_i=\xi_{k_i,l_i,j_i}$ 
   & $\sum\limits_{i=1}^{s} l_i \equiv 0(\mod n''/d)$ \\ \hline
   IV & $9|n$, $3\nmid d$, \newline $G\ne\langle A,B^{3^v}\rangle\times 
   O_{v}^{*}$ & $\varphi_i=\gamma_{k_i',l_i',j}$ 
   & $(k+1)\sum\limits_{i=1}^{s} l_i' \equiv 0(\mod n''/d)$ \\ \hline
   IV & $3|d$ & $\varphi_i=\eta_{k_i',l_i'}$ 
   & $(k+1)\sum\limits_{i=1}^{s} l_i' \equiv 0(\mod n')$ \\ \hline
   V & \phantom{0} & $\varphi_i=\iota_{k_i,l_i,j_i}$ 
   & $\sum\limits_{i=1}^{s} l_i \equiv 0(\mod n')$ \\ \hline
   VI & \phantom{0} & $\varphi_i=\varkappa_{k_i',l_i',j}$ 
   & $(k+1)\sum\limits_{i=1}^{s} l_i' \equiv 0(\mod n')$ \\
   \hline
   \end{tabular}
\end{center}
\caption{When the image of a representation is contained in $SL(N,\C)$}
\label{Tb:gorcond}
\end{table}
\begin{proof}
The theorem is a direct consequence of Theorems~\ref{T:Watanabe},
\ref{T:irreprs}, \ref{T:classisqsing}, and conditions on the parameters $n$,
$d$, etc. Fore instance, assume that $G$ is a group of type I, and
$\varphi$ is a direct sum of $s$ $d$-dimensional irreducible representations
$\pi_{k_j,l_j}$ described in the List of irreducible representations without
fixed points, $1\leq j\leq s$, $s=N/d$. From Table~\ref{Tb:reprs} we get 
the value $|\pi_{k_j,l_j}(B)|=(-1)^(d-1) e^{2\pi il_j/n'}$, $n'=n/d$. It
follows that
$$|\varphi(B)|=\prod_{j=1}^{s}(-1)^(d-1) e^{2\pi il_j/n'}=
(-1)^{s(d-1)}\exp\left({\frac{2\pi i}{n'}\sum_{j=1}^{s} l_j}\right)\,.$$
This number is equal to $1$ if and only if $(1/n')\sum l_j-(sn'(d-1))/2$ 
is integer, which is equivalent to the condition in the first row, the 
fourth column of Table~\ref{Tb:gorcond}.

Let us consider also the condition in the third row of 
Table~\ref{Tb:gorcond}. It follows from Table~\ref{Tb:reprs} that
$\varphi(B)$ has determinant $1$ if and only if 
$$6\sum_{i=1}^{s} l_i\equiv 0(\mod n')\,.$$
But recall that here $n$ is odd, divisible by $3$, but not divisible by $9$, 
and $d$ is not divisible by $3$. Thus $n'$ is odd and divisible by $3$. 
Hence we can rewrite the condition in the form given in 
Table~\ref{Tb:gorcond}. The rest of Table~\ref{Tb:gorcond} can be checked 
in the same way.
\end{proof}
\begin{remark}
By Lemma~\ref{L:conjgrps} two Gorenstein IQS $\C^N/\varphi(G)$ and 
$\C^N/\psi(G)$ are isomorphic if and only if $\varphi(G)$ and $\psi(G)$ are
conjugate subgroups of $SL(N,\C)$. One can get explicite conditions from
Theorem~\ref{T:aut}. Also note that since the numbers $n''$ and $3^v$ are 
coprime, two conditions of row III, Case $9|n$, $3\nmid d$ of 
Table~\ref{Tb:gorcond} can actually be reduced to one.
\end{remark}

To illustrate Theorem~\ref{T:classgiqsing}, let us describe complex 
Gorenstein IQS in dimensions $N$, $2\leq N\leq 7$.

$N=2$. Here we have to classify up to conjugacy finite subgroups of
$SL(2,\C)$. This case is classical, and it is well known that all such
groups are either cyclic (they belong to type I in Wolf's terminology),
binary dihedral (type I and II), binary tetrahedral (type III),
binary octahedral (type IV), or binary icosahedral (type V). The
corresponding quotients $\C^2/G$ are the Kleinian singularities $A_n$,
$D_n$, $E_6$, $E_7$, $E_8$.

$N=3$. By Theorem~\ref{T:KN} of Kurano and Nishi, all Gorenstein IQS of 
dimension $3$ are cyclic, i.~e., they are isomorphic to quotients 
$\C^3/(\Z/n)$, $n\geq 2$, where the cyclic group is generated by the matrix
$$\begin{pmatrix}
e^{2\pi il_1/n} & 0 & 0 \\
0 & e^{2\pi il_2/n} & 0 \\
0 & 0 & e^{2\pi il_3/n}
\end{pmatrix},
$$
and $(l_1,n)=(l_2,n)=(l_3,n)=1$, $l_1+l_2+l_3\equiv 0(\mod n)$ (in fact
it follows from this that $n$ must be odd).

$N=4$. Let $\C^4/\varphi(G)$ be a $4$-dimensional Gorenstein isolated
quotient singularity, as described in Theorem~\ref{T:classgiqsing}.
The representation $\varphi$ is a direct sum of irreducible representations
of dimension $q$, $q=1$, $2$, or $4$. The case $q=1$ corresponds to
cyclic IQS analogous to case $N=3$. Now let $q=2$. If $G$ is a group
of type I, then $\varphi=\pi_{k_1,l_1}\oplus\pi_{k_2,l_2}$, the 
corresponding condition of Table~\ref{Tb:gorcond} takes the form 
$l_1+l_2\equiv 0(\mod n/2)$, and $n$ must be divisible by $4$. The
corresponding singularities can be described in the following way.
Consider all ordered collections $(m,n,r,k_1,l_1,k_2,l_2)$ of $7$ positive
integers satisfying the conditions $4|n$, $r^2\equiv 1(\mod n)$, 
$((r-1)n,m)=1$, $k_i$, $l_i$ are defined modulo and coprime to $m$ and $n$ 
respectively, $i=1,2$, and $l_1+l_2\equiv 0(\mod n/2)$. Consider a group 
$G=G(m,n,r)$ of type I defined in Table~\ref{Tb:solvgrps} and its 
irreducible $2$-dimensional representations $\pi_{k_i,l_i}$, $i=1,2$,
defined in the List of irreducible representations without fixed points
(simply the List in the sequel).
Then for any such collection we have a Gorenstein isolated quotient
singularity
$$X_I(m,n,r,k_1,l_1,k_2,l_2)=
\C^4/\pi_{k_1,l_1}(G)\oplus\pi_{k_2,l_2}(G)\,.$$

Let $G$ be a group of type II. Then $\varphi=\alpha_{k_1',l_1'}\oplus
\alpha_{k_2',l_2'}$. In order to get a $2$-dimensional irreducible 
representation of it, we have to take $d=1$, thus also $A=1$ (see 
Table~\ref{Tb:solvgrps}). The corresponding condition of 
Table~\ref{Tb:gorcond} takes the form $(k+1)(l_1'+l_2')\equiv 
0(\mod n)$. The corresponding singularities can be described in the
following way. Consider all ordered collections $(n,k,l_1',l_2')$
of $4$ positive integers such that $n=2^u v$, $u\geq 2$, $k\equiv -1(2^u)$,
$k^2\equiv 1(n)$, $(l_1,n)=(l_2,n)=1$ are defined modulo $n$,
$(k+1)(l_1'+l_2')\equiv 0(n)$. Consider a group $G=G(m=1,n,r=1,l=1,k)$ 
of type II defined in Table~\ref{Tb:solvgrps} and its irreducible 
representations $\alpha_{k_i'=1,l_i'}$, $i=1,2$, defined 
in the List. Then for any such collection we have a Gorenstein isolated 
quotient singularity
$$X_{II}(n,k,l_1',l_2')=\C^4/\alpha_{1,l_1'}(G)\oplus\alpha_{1,l_2'}(G)\,.$$

Let $G$ be a group of type III. From the condition $q=2$ we again
conclude that $d=1$ and $A=1$. We have $3$ possibilities for the 
group $G$. If $9\nmid n$, then $\varphi=\nu_{k_1,l_1}\oplus\nu_{k_2,l_2}$.
The corresponding condition of Table~\ref{Tb:gorcond} takes the form
$l_1+l_2\equiv 0(n/3)$. The corresponding singularities can be
described in the following way. Consider all triples $(n,l_1,l_2)$
of positive integers such that $n$ is odd, $3|n$ but $9\nmid n$, $l_1$ and 
$l_2$ are defined modulo and coprime to $n$, and $l_1+l_2\equiv 0(n/3)$.
Consider a group $G=G(m=1,n,r=1)$ of type III defined in 
Table~\ref{Tb:solvgrps} and its irreducible representations
$\nu_{k_i=1,l_i}$, $i=1,2$, defined in the List. Then for any
such triple we have a Gorenstein isolated quotient singularity
$$X_{III}^{(1)}(n,l_1,l_2)=\C^4/\nu_{1,l_1}(G)\oplus\nu_{1,l_2}(G)\,.$$

If $9|n$, but $3\nmid d$, then $\varphi=\nu_{k_1,l_1,j_1}\oplus
\nu_{k_2,l_2,j_2}$. Here we get the following description of the
corresponding singularities. Consider all ordered collections
$(n,l_1,l_2,j_1,j_2)$ of $5$ positive integers satisfying the conditions
$n=3^v n''$, $v\geq 2$, $(l_i,n)=1$ are defined modulo $n$,
$(j_i,3)=1$ are defined modulo $3^v$, $i=1,2$, $l_1+l_2\equiv
0(n/3^v)$, $j_1+j_2\equiv 0(3^v)$. Consider a group $G=G(m=1,n,r=1)$
of type III defined in Table~\ref{Tb:solvgrps} and its irreducible
representations $\nu_{k_i=1,l_i,j_i}$, $i=1,2$, defined in the List. 
Then for any such collection we have a Gorenstein isolated quotient 
singularity
$$X_{III}^{(2)}(n,l_1,l_2,j_1,j_2)=\C^4/\nu_{1,l_1,j_1}(G)\oplus
\nu_{1,l_2,j_2}(G)\,.$$

The third possibility $3|d$ obviously does not realize here.

Assume that $G$ is a group of type IV. $2$-dimensional irreducible
representations exist only in the case $G\simeq O^*\times
\langle A,B^3\rangle$. Again there must be $d=1$ and $A=1$. We deduce
the following description of the corresponding singularities. Consider
all ordered collections $(n,k,l_1,l_2,j_1,j_2)$ of $6$ positive integers 
such that $n$ is odd, $3|n$ but $9\nmid n$, $k$ is defined modulo $n$, 
$k\equiv -1(3)$, $k^2\equiv 1(n)$, $l_i$ are defined modulo $n$, 
$(l_i,n)=1$, $j_i=\pm 1$, $i=1,2$, and $l_1+l_2\equiv 0(n/3)$. Consider a 
group $G=G(m=1,n,r=1,l=1,k)$ of type IV defined in Table~\ref{Tb:solvgrps} 
and its irreducible representations $\psi_{k_i=1,l_i,j_i}$, $i=1,2$, defined 
in the List. Then for any such collection we have a Gorenstein isolated 
quotient singularity
$$X_{IV}(n,k,l_1,l_2,j_1,j_2)=\C^4/\psi_{1,l_1,j_1}(G)\oplus
\psi_{1,l_2,j_2}(G)\,.$$

Assume that $G$ is a group of type V. We have $2$-dimensional irreducible
representations only in the case $d=1$, $A=1$. The corresponding
singularities can be described in the following way. Consider all ordered
collections $(n,l_1,l_2,j_1,j_2)$ of $5$ positive integers such that 
$(n,30)=1$, $(l_i,n)=1$ are defined modulo $n$, $j_i=\pm 1$, $i=1,2$, and 
$l_1+l_2\equiv 0(n)$. Consider a group $G=G(m=1,n,r=1)$ of type V defined
in Theorem~\ref{T:nonsolvgrps} and its irreducible representations
$\iota_{k_i=1,l_i,j_i}$, $i=1,2$, defined in the List. Then for any such
collection we have a Gorenstein isolated quotient singularity 
$$X_V(n,l_1,l_2,j_1,j_2)=\C^4/\iota_{1,l_1,j_1}(G)\oplus
\iota_{1,l_2,j_2}(G)\,.$$

There are no $2$-dimensional irreducible representations of groups 
of type VI.

Now let $q=4$, i.~e., the group $G$ acts on $\C^4$ via an irreducible
representation $\varphi$. Irreducible representations of groups of
types I -- VI with determinants $1$ are described in 
Theorem~\ref{T:irreprs} and Table~\ref{Tb:reprs}. We obtain the following
singularities. Let $G$ be a group of type I. Consider all triples $(m,r,k)$
of positive integers such that $m$ is odd, $r^4\equiv 1(m)$, $(r-1,m)=1$,
and $k$ is defined modulo $m$, $(k,m)=1$. Consider a group $G=G(m,n=8,r)$ 
of type I defined in Table~\ref{Tb:solvgrps} and its irreducible 
representation $\pi_{k,l=2}$ defined in the List. Then for any such triple 
we have a Gorenstein isolated quotient singularity
$$X_I(m,r,k)=\C^4/\pi_{k,2}(G)\,.$$

Let $G$ be a group of type II. Consider all ordered collections 
$(m,n,r,l,k',l')$ of $6$ positive integers such that $(m,n,r,l)$ satisfy 
the conditions of Table~\ref{Tb:solvgrps} Type II with $d=2$, $k=-1$, $k'$ 
is defined modulo $m$, $l'$ is defined modulo $n$, $(k',m)=(l',n)=1$. 
Consider a group $G=G(m,n,r,l,k=-1)$ of type II defined in 
Table~\ref{Tb:solvgrps} and its irreducible representation $\alpha_{k',l'}$ 
defined in the List. Then for any such collection we have a Gorenstein 
isolated quotient singularity
$$X_{II}(m,n,r,l,k',l')=\C^4/\alpha_{k',l'}(G)\,.$$

There are no irreducible representations of dimension $4$ with determinant
$1$ of groups of type III.

Let $G$ be a group of type IV. We get the following singularities. Consider
a pair $(n,l')$ of positive integers such that $n$ is odd, $3|n$ but
$9\nmid n$, $(l',n)=1$ is defined modulo $n$. Consider a group $G=
G(m=1,n,r=1,l=1,k=-1)$, $G\ne\langle A,B^3\rangle\times O^*$, of type IV 
defined in Table~\ref{Tb:solvgrps} and in the List, and its irreducible
representation $\gamma_{k'=1,l'}$ defined in the List. Then for any such
pair we have a Gorenstein isolated quotient singularity
$$X_{IV}^{(1)}(n,l')=\C^4/\gamma_{1,l'}(G)\,.$$

For any $v>1$ and $k$, $1\leq k<3^v$, $k\equiv 1(3)$, we have a Gorenstein
isolated quotient singularity
$$X_{IV}^{(2)}(v,k)=\C^4/o_k(O_{v}^{*})\,,$$
where $O_{v}^{*}$ is the generalized octahedral group and $o_k$ is its
irreducible representation defined in Lemma~\ref{L:Ogrpreprs}.

Consider all ordered collections $(n,k,l',j)$ of $4$ positive integers
such that the pair $(n,k)$ satisfies the conditions of 
Table~\ref{Tb:solvgrps}, Type IV with $m=1$, $r=1$, $l=1$, $n=3^v n''$,
$v\geq 2$, $l'$ is defined modulo $n$, $(l',n)=1$, $1\leq j<3^v$,
$j\equiv 1(3)$. Consider a group $G=G(m=1,n,r=1,l=1,k)$, $G\ne
\langle A,B^{3^v}\rangle\times O_{v}^{*}$, of type IV defined in 
Table~\ref{Tb:solvgrps} and in the List, and its irreducible representation 
$\gamma_{k'=1,l',j}$ defined in the List. Then for any such collection we
have a Gorenstein isolated quotient singularity
$$X_{IV}^{(3)}(n,k,l',j)=\C^4/\gamma_{1,l',j}(G)\,.$$

There are no irreducible representations of dimension $4$ with determinant
$1$ of groups of type V.

Let $G$ be a group of type VI. Consider all triples $(n,l',j)$ satisfying
the conditions $(n,30)=1$, $l'$ is defined modulo $n$, $(l',n)=1$,
$j=\pm 1$. Consider a group $G=G(m=1,n,r=1,l=1,k=-1)$ of type VI defined
in Theorem~\ref{T:nonsolvgrps} and its irreducible representation
$\varkappa_{k'=1,l',j}$ defined in the List. Then for any such triple
we have a Gorenstein isolated quotient singularity
$$X_{VI}(n,l',j)=\C^4/\varkappa_{1,l',j}(G)\,.$$

$N=5$. By Theorem~\ref{T:KN}, in this case we have only cyclic
quotient singularities $\C^5/(\Z/n)$, where the group is generated by a
diagonal matrix with $e^{2\pi il_j/n}$, $(l_j,n)=1$, $j=1,\dots,5$, on the 
diagonal, $l_1+\dots+l_5\equiv 0(\mod n)$.

$N=6$. Again we have to consider all divisors $q$ of $6$. If $q=1$,
we get cyclic quotient singularities. Let $q=2$. In this case we get
singularities of the form $\C^6/\varphi(G)$, where $G$ is a group of
one of the types I -- VI, and $\varphi$ is a direct sum of $3$
$2$-dimensional irreducible representations of $G$ without fixed points,
satisfying the Gorenstein condition of Table~\ref{Tb:gorcond}. For
example, let $G$ be a group of type I. Consider all ordered collections
$(m,n,r,k_1,l_1,k_2,l_2,k_3,l_3)$ of $9$ positive integers, where
$(m,n,r)$ satisfy the conditions of Table~\ref{Tb:solvgrps}, type I,
$r^2\equiv 1(m)$, $k_i$ are defined modulo $m$, $(k_i,m)=1$, $l_i$ are
defined modulo $n$, $(l_i,n)=1$, and $2(l_1+l_2+l_3)\equiv 3(n/2)(\mod n)$.
It is not difficult to see that the last condition is equivalent to
the following: $n/4$ is odd and divides $l_1+l_2+l_3$. Consider a group
$G=G(m,n,r)$ of type I defined in Table~\ref{Tb:solvgrps} and its
irreducible representations $\pi_{k_i,l_i}$, $i=1,2,3$, defined in the
List. Then for any such collection we have a Gorenstein isolated
singularity
$$X_I(m,n,r,k_1,l_1,k_2,l_2,k_3,l_3)=\C^6/\pi_{k_1,l_1}(G)\oplus
\pi_{k_2,l_2}(G)\oplus\pi_{k_3,l_3}(G)\,.$$
We leave to an interested reader to state a more precise description of
the singularities corresponding to types II -- V. Note that groups
of type VI does not give any singularities in this case.

Assume that $q=3$. In this case all singularities are produced by
groups of type I, since only they have irreducible representations
without fixed points of odd dimension. Consider all ordered collections
$(m,n,r,k_1,l_1,k_2,l_2)$ of $7$ positive integers, where $(m,n,r)$
satisfy conditions of Table~\ref{Tb:solvgrps}, type I, $r^3\equiv 1(m)$,
$k_i$ are defined modulo $m$, $(k_i,m)=1$, $l_i$ are defined modulo $n$,
$(l_i,n)=1$, $i=1,2$, and $l_1+l_2\equiv 0(n/3)$. Consider a group
$G=G(m,n,r)$ of type I defined in Table~\ref{Tb:solvgrps} and its
irreducible representations $\pi_{k_i,l_i}$, $i=1,2$, defined in the List.
Then for any such collection we have a Gorenstein isolated quotient 
singularity
$$X_I(m,n,r,k_1,l_1,k_2,l_2)=\C^6/\pi_{k_1,l_1}(G)\oplus
\pi_{k_2,l_2}(G)\,.$$

There are no irreducible representations without fixed points of
dimension $6$ of groups of types I -- VI.

$N=7$. By Theorem~\ref{T:KN}, there are only cyclic Gorenstein IQS in
dimension $7$.

\end{document}